\documentclass{amsart}

\usepackage{amsmath}
\usepackage{amssymb}
\usepackage{amsthm}
\usepackage{amscd}

\usepackage[all]{xy}

\newtheorem{thm}{Theorem}[section]
\newtheorem{prop}[thm]{Proposition}
\newtheorem{lem}[thm]{Lemma}
\newtheorem{cor}[thm]{Corollary}
\newtheorem{ithm}{Theorem}

\theoremstyle{definition}
\newtheorem{dfn}[thm]{Definition}

\theoremstyle{remark}
\newtheorem{rem}{Remark}
\newtheorem*{acknowledgments}{Acknowledgments}

\newcommand{\C}{\mathbb{C}}

\newcommand{\Z}{\mathbb{Z}}

\newcommand{\A}{\mathfrak{A}}

\newcommand{\KF}{K\!F}

\newcommand{\Hom}{\mathrm{Hom}}
\newcommand{\Ker}{\mathrm{Ker}}
\newcommand{\pt}{\mathrm{pt}}

\newcommand{\cC}{\mathcal{C}}
\renewcommand{\H}{\mathcal{H}}
\newcommand{\F}{\mathcal{F}}
\newcommand{\U}{\mathcal{U}}

\newcommand{\B}{\mathcal{B}}
\newcommand{\K}{\mathcal{K}}

\newcommand{\HF}{\mathcal{HF}}
\newcommand{\cKF}{\mathcal{KF}}
\renewcommand{\P}{\mathcal{P}}

\newcommand{\tbigoplus}{\textstyle \bigoplus}
\newcommand{\tcoprod}{{\textstyle \coprod}}

\title{Twisted $K$-theory and finite-dimensional approximation}

\author[K. Gomi]{Kiyonori Gomi}

\address{
The Department of Mathematics,
The University of Texas at Austin,
1 University Station C1200, Austin, TX 78712-0257 
USA.}

\email{kgomi@math.utexas.edu}



\date{}

\begin{document}

\begin{abstract}
We provide a finite-dimensional model of the twisted $K$-group twisted by any degree three integral cohomology class of a CW complex. One key to the model is Furuta's generalized vector bundle, and the other is a finite-dimensional approximation of Fredholm operators.
\end{abstract}

\maketitle



\section{Introduction}
\label{sec:introduction}

Since the work of Atiyah and Hirzebruch \cite{A-H}, $K$-theory has been recognized as a fundamental notion in topology and geometry. Twisted $K$-theory is a variant of $K$-theory originating from the works of Donovan-Karoubi \cite{D-K} and Rosenberg \cite{R}. Much focus is on twisted $K$-theory recently, due to applications, for example, to $D$-brane charges (\cite{K,W}), Verlinde algebras \cite{F-H-T} and quantum Hall effects \cite{CHMM}.

\medskip

As is well-known, the $K$-group $K(X)$ of a compact space $X$ admits various formulations. The standard formulation of $K(X)$ uses finite dimensional vector bundles on $X$. One can also formulate $K(X)$ by using a $C^*$-algebra as well as the space of Fredholm operators. To define twisted $K$-theory, we usually appeal to the latter two formulations above, involving some infinite dimensions. 

\medskip

$K$-theory enjoys numerous applications to topology and geometry because of its realization by means of vector bundles. To give a similar realization of twisted $K$-theory seems to be an interesting problem to be studied not only for better understanding but also for further applications.

\medskip

So far, as a partial answer to the problem, \textit{twisted vector bundles} or \textit{bundle gerbe $K$-modules} \cite{BCMMS} are utilized to realize the twisted $K$-group whose ``twisting'' satisfies a condition. The condition is that the degree three integral cohomology class corresponding to the twisting is of finite order. A complete answer to this realization problem, valid for twistings corresponding to any degree three integral cohomology classes, was known by Andr\'e Henriques. The aim of the present paper is to give another complete answer by generalizing the following result announced in \cite{G}:

\begin{ithm} \label{ithm:main}
Let $X$ be a CW complex, $P$ a principal bundle over $X$ whose structure group is the projective unitary group of a separable Hilbert space of infinite dimension, and $K_P(X)$ the twisted $K$-group. We write $\KF_P(X)$ for the homotopy classes of $P$-twisted ($\Z_2$-graded) vectorial bundles over $X$. Then there exists a natural isomorphism $\alpha : \ K_P(X) \longrightarrow \KF_P(X)$.
\end{ithm}

\medskip

The notion of \textit{vectorial bundle} is a generalization of the notion of vector bundles due to Mikio Furuta \cite{F1}. Vectorial bundles realize the ordinary $K$-group $K(X)$, and arise as finite-dimensional geometric objects approximating families of Fredholm operators. We can think of the approximation as a linear version of the finite-dimensional approximation of the Seiberg-Witten equations \cite{F2}. Since $K_P(X)$ consists of certain families of Fredholm operators, a twisted version of vectorial bundles provides a suitable way to realize twisted $K$-theory.

\medskip

As a simple application of Theorem \ref{ithm:main}, we can generalize some notions of 2-vector bundles \cite{B-D-R,B}. The notion of 2-vector bundles in the sense of Brylinski \cite{B} uses the category of vector bundles, and a 2-vector bundle of rank 1 reproduces the category of twisted vector bundles, so that the twisted $K$-group whose twisting corresponds to a degree three integral cohomology class of finite order. By using the category of vectorial bundles instead, we get a proper generalization of Brylinski's 2-vector bundles. This generalization reproduces the category of twisted vectorial bundles, and hence the twisted $K$-group with any twisting. A similar replacement may generalize 2-vector bundles of Baas, Dundas and Rognes \cite{B-D-R}, which they studied in seeking for a geometric model of elliptic cohomology.

\medskip

Also, Theorem \ref{ithm:main} allows us to construct Chern characters of 
twisted $K$-classes in a purely finite-dimensional manner \cite{G-Y}.


\medskip

In a word, the proof of Theorem \ref{ithm:main} is a comparison of cohomology theories: as is well-known, the twisted $K$-group $K_P(X)$ fits into a certain generalized cohomology theory $K_P^*(X, Y)$. The group $\KF_P(X)$ also fits into a similar cohomology theory $\KF_P^*(X, Y)$, and the homomorphism $\alpha : K_P(X) \to \KF_P(X)$ extends to a natural transformation between these two cohomology theories. Then, appealing to a standard method in algebraic topology, we compare these cohomology theories to show their equivalence.

\medskip

According to the outline of the proof above, this paper is organized as follows. In Section \ref{sec:twisted_K}, we review a definition of twisted $K$-theory and a construction of twisted $K$-cohomology $K^*_P(X, Y)$. In Section \ref{sec:vectorial_bundle}, we introduce the notion of ($\Z_2$-graded) vectorial bundles and its twisted version. In Section \ref{sec:cohomology_KF}, we construct the cohomology theory $\KF^*_P(X, Y)$. In Section \ref{sec:fda}, we construct the natural transformation between $K^*_P(X, Y)$ and $\KF^*_P(X, Y)$. A key to the construction is a finite-dimensional approximation of a family of Fredholm operators. After a study of the natural transformation, we compare the cohomology theories to derive Theorem \ref{ithm:main}. Finally, in Appendix, proof of Furuta's results, crucial to the present paper, are culled from the Japanese textbook \cite{F1} for convenience.

\bigskip

\begin{acknowledgments}
I benefited a lot from discussion with M. Furuta on various stage of the work, and I am grateful to him. I thank T. Moriyama for suggestions regarding the proof of Theorem \ref{ithm:main}. I also thank A. Henriques for suggestions about the proof and for comments on a draft. I am indebted to D. Freed for useful discussion. The author's research is supported by a JSPS Postdoctoral Fellowship for Research Abroad.
\end{acknowledgments}


\section{Twisted $K$-theory}
\label{sec:twisted_K}

We here review twisted $K$-cohomology theory, following \cite{A-Se,C-W} mainly.


\subsection{Review of twisted $K$-theory}

Let $PU(\H) = U(\H)/U(1)$ be the projective unitary group of a separable Hilbert space $\H$ of infinite-dimension. We topologize $PU(\H)$ by using the compact-open topology in the sense of \cite{A-Se}. Let $\F(\H)$ be the set of bounded linear operators $A : \H \to \H$ such that $A^*A - 1$ and $AA^* - 1$ are compact operators:
$$
\F(\H) = \{ A : \H \to \H |\ A^*A - 1, AA^* - 1 \in \mathcal{K}(\H) \}.
$$
Note that $\F(\H)$ is a subset of the space of Fredholm operators on $\H$. We induce a topology on $\F(\H)$ by using the map 
$$
\begin{array}{ccc}
\F(\H) & \longrightarrow &
\mathcal{B}(\H)_\mathrm{co} \times \mathcal{B}(\H)_\mathrm{co} \times 
\mathcal{K}(\H)_\mathrm{norm} \times \mathcal{K}(\H)_\mathrm{norm}, \\
 & & \\
A & \mapsto & (A, \ A^*, \ A^*A-1, \ AA^*-1)
\end{array}
$$
where $\mathcal{B}(\H)_\mathrm{co}$ is the space of bounded linear operators $\mathcal{B}(\H)$ topologized by the compact-open topology, and $\mathcal{K}(\H)_\mathrm{norm}$ is the space of compact operators $\mathcal{K}(\H)$ topologized by the usual operator norm. Then $\F(\H)$ is a representing space for $K$-theory, and $PU(\H)$ acts continuously on $\F(\H)$ by conjugation (\cite{A-Se}). 

In this paper, the ``twist'' in twisted $K$-theory is given by a principal $PU(\H)$-bundle. For a principal $PU(\H)$-bundle $P \to X$ given, the conjugate action gives the associated bundle $P \times_{Ad} \F(\H) \to X$ whose fiber is $\F(\H)$.

\begin{dfn} \label{def:twisted_K}
Let $X$ be a compact Hausdorff space, and $P \to X$ a principal $PU(\H)$-bundle. We define the twisted $K$-group $K_P(X)$ to be the group consisting of fiberwise homotopy classes of the sections of $P \times_{Ad} \F(\H) \to X$:
$$
K_P(X) = \Gamma(X, P \times_{Ad}\F(\H))/\mathrm{homotopy},
$$
where the addition in $K_P(X)$ is given by fixing an isomorphism $\H \oplus \H \cong \H$.
\end{dfn}

\begin{rem}
As in \cite{A-Se}, we can consider a more refined ``twist'' by introducing a $\Z_2$-grading to the Hilbert space and using unitary transformations of degree 1. However, the present paper does not cover the case.
\end{rem}

\begin{rem}
In \cite{A-Se}, a projective space bundle plays the role of a ``twisting''. Since the structure group of the projective space bundle is $PU(\H)$, Definition \ref{def:twisted_K} gives the same twisted $K$-group as that in \cite{A-Se}.
\end{rem}

\begin{rem}
Instead of the compact-open topology, we can also work with the topology on $PU(\H)$ given by the operator norm. In this case, the formulation of twisted $K$-theory uses the space of (bounded) Fredholm operators on $\H$ equipped with the operator norm topology, instead of $\F(\H)$. An advantage of the compact-open topology, other than that pointed out in \cite{A-Se}, is that it simplifies some argument in Subsection \ref{subsec:fda}.
\end{rem}


\subsection{Review of twisted $K$-cohomology}

We formulate twisted $K$-cohomology as a certain generalized cohomology.

\medskip

We write $\cC$ for the category of CW pairs: an object in $\cC$ is a pair $(X, Y)$ consisting of a CW complex $X$ and its subcomplex $Y$. A morphism $f : (X', Y') \to (X, Y)$ is a continuous map $f : X' \to X$ such that $f(Y') \subset Y$. We also write $\widehat{\cC}$ for the category of CW pairs equipped with $PU(\H)$-bundles:  an object $(X, Y; P)$ in $\widehat{\cC}$ consists of a CW pair $(X, Y) \in \cC$ and a principal $PU(\H)$-bundle $P \to X$. A morphism $(f, F) : (X', Y'; P') \to (X, Y; P)$ consists of a morphism $f : (X', Y') \to (X, Y)$ in $\cC$ and a bundle map $F : P' \to P$ covering $f$. A CW complex $X$ equipped with a principal $PU(\H)$-bundle $P \to X$ will be identified with $(X, \emptyset; P) \in \widehat{\cC}$.

\medskip

Let $(X, Y) \in \cC$ be a CW pair and $P \to X$ a principal $PU(\H)$-bundle. The \textit{support} of a section $\mathbb{A} \in \Gamma(X, P \times_{Ad} \F(\H))$ is defined to be the closure of the set consisting of the points at which $\mathbb{A}$ is not invertible:
$$
\mathrm{Supp} \mathbb{A} = 
\overline{
\{ x \in X |\ \mbox{$\mathbb{A}_x : \H \to \H$ is not invertible} \}
}.
$$
We define $K_P(X, Y)$ by using sections $\mathbb{A} \in \Gamma(X, P \times_{Ad} \F(\H))$ such that $\mathrm{Supp} \mathbb{A} \cap Y = \emptyset$. In $K_P(X, Y)$, two sections $\mathbb{A}_0$ and $\mathbb{A}_1$ are identified if they are connected by a section $\tilde{\mathbb{A}} \in \Gamma(X \times I, (P \times I) \times_{Ad} \F(\H))$ such that $\mathrm{Supp}{\tilde{\mathbb{A}}} \cap (Y \times I) = \emptyset$, where $I$ is the interval $[0, 1]$. We then define the twisted $K$-cohomology groups $K_P^{-n}(X, Y)$ as follows:
$$
K_P^{-n}(X, Y) = 
\left\{
\begin{array}{lc}
K_{P \times I^n}(X\times I^n, Y \times I^n \cup X \times \partial I^n), &
(n \ge 0) \\
K_P^{n}(X, Y). &
(n < 0)
\end{array}
\right.
$$
Clearly, a morphism $(f, F) : (X', Y'; P') \to (X, Y; P)$ induces a homomorphism $(f, F)^* : K_P^n(X, Y) \to K_{P'}^n(X', Y')$ for all $n \in \Z$. In the case of $P' = f^*P$, we simply write $f^* :  K_P^n(X, Y) \to K_{f^*P}^n(X', Y')$ for the homomorphism induced from $(f, \widehat{f})$, where $\widehat{f} : f^*P \to P$ is the canonical bundle map covering $f$.

Now, we summarize basic properties of twisted $K$-cohomology theory (\cite{A-Se,C-W,F-H-T1}):

\begin{prop} \label{prop:twisted_K_cohomology}
The assignment of $\{ K_P^n(X, Y) \}_{n \in \Z}$ to $(X, Y; P) \in \widehat{\cC}$ has the following properties:
\begin{enumerate}
\item (Homotopy axiom) 
If $(f_i, F_i) : (X', Y'; P') \to (X, Y; P)$, ($i = 0, 1$) are homotopic, then the induced homomorphisms coincide: $(f_0, F_0)^* = (f_1, F_1)^*$.

\item (Excision axiom)
For subcomplexs $A, B \subset X$, the inclusion map induces the isomorphism:
$$
K_{P|_{A \cup B}}^n(A \cup B, B) \cong K_{P|_{A}}^n(A, A \cap B). \quad (n \in \Z)
$$

\item (Exactness axiom)
There is the natural long exact sequence:
$$
\cdots \to
K^{n-1}_{P|_Y}(Y) \overset{\delta_{n-1}}{\to}
K^{n}_P(X, Y) \to
K^{n}_P(X) \to
K^{n}_{P|_Y}(Y) \overset{\delta_{n}}{\to} \cdots.
$$

\item (Additivity axiom)
For a family $\{ (X_\lambda, Y_\lambda; P_\lambda) \}_{\lambda \in \Lambda}$ in $\widehat{\cC}$, the inclusion maps $X_\lambda \to \coprod_\lambda X_\lambda$ induce the natural isomorphism:
$$
K_{\coprod_\lambda P_\lambda}^{-n}
(\tcoprod_\lambda X_\lambda, \tcoprod_\lambda Y_\lambda)
\cong
\prod_\lambda 
K_{P_\lambda}^{-n}(X_\lambda, Y_\lambda). \quad (n \in \Z)
$$

\item (Bott periodicity)
There is the natural isomorphism:
$$
\beta_n : \ K_P^n(X, Y) \longrightarrow K_P^{n-2}(X, Y). \quad (n \in \Z)
$$

\end{enumerate}
\end{prop}

The homotopy axiom and the additivity axiom are clear. The excision axiom is due to the fact that the set of invertible operators in $\F(\H)$ is contractible \cite{A-Se}. The periodicity is a consequence of the homotopy equivalence $\F(\H) \simeq \Omega^2\F(\H)$, (\cite{A-Se,A-Si}). The exactness axiom follows essentially from the cofibration sequence:
$$
\cdots \longleftarrow
\Sigma^2 Y \longleftarrow
\Sigma (X/Y) \longleftarrow
\Sigma X \longleftarrow 
\Sigma Y \longleftarrow
X/Y \longleftarrow
X \longleftarrow
Y,
$$
where $\Sigma$ stands for the reduced suspension. Noting the homotopy equivalence $\Sigma^n(X \coprod \pt) \simeq X \times I^n / (Y \times I^n \cup X \times \partial I^n)$, we obtain the non-positive part of the long exact sequence for a pair $(X, Y)$ in a way similar to that used in \cite{A}. Then we get the positive part by using the Bott periodicity. A similar construction of the exact sequence will be performed in Subsection \ref{subsec:exactness_axiom} in our finite-dimensional model.

\begin{rem}
The definition of $K_P(X,Y)$ in  \cite{C-W} is equivalent to that in this paper, because of the definition of the support of $\mathbb{A} \in \Gamma(X, P \times_{Ad} \F(\H))$. The definition of $K^{-n}_P(X)$ in \cite{A-Se,C-W}, which utilizes sections of the bundle $P \times_{Ad} \Omega^n\F(\H)$ over $X$, is also equivalent to our definition.
\end{rem}


\section{Vectorial bundle}
\label{sec:vectorial_bundle}

We here introduce Furuta's generalized vector bundles \cite{F1} as \textit{vectorial bundles}.  Our formulation differs slightly from the original formulation in \cite{F1}. Twisted vectorial bundles are also introduced in this section.


\subsection{Vectorial bundle}
\label{subsec:vectorial_bundle}

\begin{dfn} \label{dfn:category_HF}
Let $X$ be a topological space. For a subset $U \subset X$, we define the category $\HF(U)$ as follows. An object in $\HF(U)$ is a pair $(E, h)$ consisting of a $\Z_2$-graded Hermitian vector bundle $E \to U$ of finite rank and a Hermitian map $h : E \to E$ of degree 1. The homomorphisms in $\HF(U)$ are defined by
$$
\Hom_{\HF(U)}((E, h), (E', h')) =
\{ \phi : E \to E' |\ \mbox{degree $0$}, \ \phi h = h' \phi \}/\fallingdotseq,
$$
where $ \ \fallingdotseq \ $ stands for an equivalence relation. That $\phi \fallingdotseq \phi'$ means:
\begin{quote}
For each point $x \in U$, there are a positive number $\mu > 0$ and an open subset $V \subset U$ containing $x$ such that: for all $y \in V$ and $\xi \in (E, h)_{y, < \mu}$, we have $\phi(\xi) = \phi'(\xi)$.
\end{quote}
In the above, we put
$$
(E, h)_{y, < \mu} 
= 
\bigoplus_{\lambda < \mu} \Ker(h_y^2 - \lambda)
=
\bigoplus_{\lambda < \mu}
\{ \xi \in E_y |\ h_y^2 \xi = \lambda \xi \}.
$$
\end{dfn}

By abuse of notation, we just write $\phi$ for the equivalence class $[\phi]$ of a map $\phi : (E, h) \to (E', h')$ in $\Hom_{\HF(U)}((E, h), (E', h'))$. For a subset $V \subset U$, the restriction $(E, h) \mapsto (E, h)|_V$ defines a functor $\HF(U) \to \HF(V)$, which composes properly for a smaller subset in $V$.

\begin{dfn} \label{dfn:vectorial_bundle}
For a space $X$, we define the category $\cKF(X)$ as follows. 

\begin{enumerate}
\item
An object $(\U, (E_\alpha, h_\alpha), \phi_{\alpha \beta})$ in $\cKF(X)$ consists of an open cover $\U = \{ U_\alpha \}_{\alpha \in \A}$ of $X$, objects $(E_\alpha, h_\alpha)$ in $\HF(U_\alpha)$, and homomorphisms $\phi_{\alpha \beta} : (E_\beta, h_\beta) \to (E_\alpha, h_\alpha)$ in $\HF(U_{\alpha \beta})$ such that:
$$
\begin{array}{rcll}
\phi_{\alpha \beta} \phi_{\beta \alpha} & = & 1 & 
\mbox{in $\HF(U_{\alpha \beta})$}; \\
\phi_{\alpha \beta} \phi_{\beta \gamma} & = & \phi_{\alpha \gamma} &
\mbox{in $\HF(U_{\alpha \beta \gamma})$},
\end{array}
$$
where $U_{\alpha \beta} = U_\alpha \cap U_\beta$ and $U_{\alpha \beta \gamma} = U_\alpha \cap U_\beta \cap U_\gamma$ as usual. We call an object in $\cKF(X)$ a \textit{vectorial bundle} over $X$. 

\item
A homomorphism $(\{ U'_{\alpha'} \}, (E'_{\alpha'}, h'_{\alpha'}), \phi'_{\alpha' \beta'}) \to (\{ U_\alpha \}, (E_\alpha, h_\alpha), \phi_{\alpha \beta})$ consists of homomorphisms $\psi_{\alpha \alpha'} : (E'_{\alpha'}, h'_{\alpha'}) \to (E_\alpha, h_\alpha)$ in $\HF(U_\alpha \cap U'_{\alpha'})$ such that the following diagrams commute in $\HF(U_\alpha \cap U'_{\alpha'} \cap U_{\beta'})$ and $\HF(U_\alpha \cap U_\beta \cap U'_{\alpha'})$, respectively.
$$
\xymatrix{
E'_{\alpha'} \ar[r]^{\psi_{\alpha \alpha'}} & E_\alpha \\
E'_{\beta'} \ar[ru]_{\psi_{\alpha \beta'}} \ar[u]^{\phi'_{\alpha' \beta'}} &
}
\quad \quad \quad \quad 
\xymatrix{
 & E_\alpha \\
E'_{\beta'} \ar[ru]^{\psi_{\alpha \beta'}} \ar[r]_{\psi_{\beta \beta'}} &
E_\beta \ar[u]_{\phi_{\alpha \beta}} 
}
$$
\end{enumerate}
\end{dfn}

An \textit{isomorphism} of vectorial bundles is a homomorphism in $\cKF(X)$ admitting an inverse. Vectorial bundles $\mathbb{E}_0$ and $\mathbb{E}_1$ over $X$ are said to be \textit{isomorphic} if there exists an isomorphism $\mathbb{E}_0 \to \mathbb{E}_1$. To indicate the relationship, we will write $\mathbb{E}_0 \cong \mathbb{E}_1$. Vectorial bundles $\mathbb{E}_0$ and $\mathbb{E}_1$ are said to be \textit{homotopic} if there exists $\tilde{\mathbb{E}} \in \cKF(X \times I)$ such that $\tilde{\mathbb{E}}|_{X \times \{ i \}} \cong \mathbb{E}_i$ for $i = 0, 1$. We will write $[ \mathbb{E} ]$ for the homotopy class of a vectorial bundle $\mathbb{E} \in \cKF(X)$. 

\begin{lem}[\cite{F1}]
Let $\KF(X)$ be the homotopy classes of vectorial bundles on $X$. Then $\KF(X)$ is an abelian group.
\end{lem}

\begin{proof}
The addition in $\KF(X)$ is given by the direct sum of vector bundles, and the inverse by reversing the $\Z_2$-grading in vector bundles. Then the present lemma will be clear, except the consistency of the definition of the inverse. To see it, we define $(F, \eta) \in \HF(I)$ by taking $F = F^0 \oplus F^1$ to be $F^i = I \times \C$ and $\eta : F \to F$ to be $\eta = \left( \begin{array}{cc} 0 & t \\ t & 0 \end{array} \right)$. We multiply $(E, h) \in \HF(X)$ by $(F, \eta)$ to get $(E \otimes F, h \otimes \mathrm{id}_F + \epsilon \otimes \eta) \in \HF(X \times I)$, where $\epsilon : E \to E$ acts on the even part $E^0$ of $E = E^0 \oplus E^1$ by $1$ and the odd part $E^1$ by $-1$. Then, as a homotopy, the object above connects the trivial object in $\HF(X)$ with $(E, h) \oplus (E^\vee, h^\vee)$, where $(E^\vee, h^\vee)$ stands for $(E, h)$ with its $\Z_2$-grading reversed. We can readily globalize this construction, so that the inverse is well-defined. 
\end{proof}

For a $\Z_2$-graded vector bundle $E$ over $X$, we can construct a vectorial bundle over $X$ by taking an open cover $\U$ of $X$ to be $X$ itself and a Hermitian map $h : E \to E$ of degree 1 to be $h = 0$. This construction of vectorial bundles induces a well-defined homomorphism $K(X) \to \KF(X)$. The following result of Furuta will be used in Subsection \ref{subsec:property_of_fda}, and its proof is included in Appendix.

\begin{thm}[\cite{F1}] \label{thm:Furuta}
If $X$ is compact, then $K(X) \to \KF(X)$ is bijective.
\end{thm}

\begin{rem}
As Definition \ref{dfn:category_HF} works without $\Z_2$-grading, the vectorial bundles in Definition \ref{dfn:vectorial_bundle} should be called \textit{$\Z_2$-graded} vectorial bundles. However, we drop the adjective ``$\Z_2$-graded'', since ungraded ones will not appear in this paper.
\end{rem}


\subsection{Twisted vectorial bundle}

\begin{dfn}
Let $X$ be a topological space, $P \to X$ a principal $PU(\H)$-bundle, and $U \subset X$ a subset.

\begin{enumerate}

\item[(a)]
We define the category $\P(U)$ as follows. The objects in $\P(U)$ consist of sections $s : U \to P|_U$. The morphisms in $\P(U)$ are defined by
$$
\Hom_{\P(U)}(s, s') = \{ g : U \to U(\H) |\ s' \pi(g) = s \},
$$
where $\pi : PU(\H) \to U(\H)$ is the projection. The composition of morphisms is defined by the pointwise multiplication.

\item[(b)]
We define the category $\HF_P(U)$ as follows. The objects in $\HF_P(U)$ are the same as those in $\P(U) \times \HF(U)$:
$$
\mathrm{Obj} (\HF_P(U)) =
\mathrm{Obj} (\P(U)) \times \mathrm{Obj} (\HF(U)).
$$
The homomorphisms in $\HF_P(U)$ are defined by:
\begin{multline*}
Hom_{\HF_P(U)}((s, (E, h)), (s', (E', h'))) \\
=
\Hom_{\P(U)}(s, s') \times \Hom_{\HF(U)}((E, h), (E', h')) / \sim,
\end{multline*}
where the equivalence relation $\sim$ identifies $(g, \phi)$ with $(g \zeta, \phi \zeta)$ for any $U(1)$-valued map $\zeta : U \to U(1)$.
\end{enumerate}
\end{dfn}

\begin{dfn} \label{dfn:twisted_vectorial_bundle}
Let $X$ be a paracompact space, and $P \to X$ a principal $PU(\H)$-bundle. We define the category $\cKF_P(X)$ as follows. 

\begin{enumerate}
\item
An object $(\U, \mathcal{E}_\alpha, \Phi_{\alpha \beta})$ in $\cKF(X)$ consists of an open cover $\U = \{ U_\alpha \}_{\alpha \in \A}$ of $X$, objects $\mathcal{E}_\alpha$ in $\HF_P(U_\alpha)$, and homomorphisms $\Phi_{\alpha \beta} : \mathcal{E}_\beta \to \mathcal{E}_\alpha$ in $\HF_P(U_{\alpha \beta})$ such that:
$$
\begin{array}{rcll}
\Phi_{\alpha \beta} \Phi_{\beta \alpha} & = & 1 & 
\mbox{in $\HF_P(U_{\alpha \beta})$}; \\
\Phi_{\alpha \beta} \Phi_{\beta \gamma} & = & \Phi_{\alpha \gamma} &
\mbox{in $\HF_P(U_{\alpha \beta \gamma})$}.
\end{array}
$$
We call an object in the category $\cKF_P(X)$ a \textit{twisted vectorial bundle} over $X$ twisted by $P$, or a \textit{$P$-twisted vectorial bundle} over $X$. 

\item
A homomorphism $(\{ U'_{\alpha'} \}, \mathcal{E}'_{\alpha'}, \Phi'_{\alpha' \beta'}) \to (\{ U_\alpha \}, \mathcal{E}_\alpha, \Phi_{\alpha \beta})$ consists of homomorphisms $\Psi_{\alpha \alpha'} : \mathcal{E}'_{\alpha'} \to \mathcal{E}_\alpha$ in $\HF_P(U_\alpha \cap U'_{\alpha'})$ such that the following diagrams commute in $\HF_P(U_\alpha \cap U'_{\alpha'} \cap U_{\beta'})$ and $\HF_P(U_\alpha \cap U_\beta \cap U'_{\alpha'})$, respectively.
$$
\xymatrix{
\mathcal{E}'_{\alpha'} 
\ar[r]^{\Psi_{\alpha \alpha'}} & 
\mathcal{E}_\alpha \\
\mathcal{E}'_{\beta'} 
\ar[ru]_{\Psi_{\alpha \beta'}} \ar[u]^{\Phi'_{\alpha' \beta'}} &
}
\quad \quad \quad \quad 
\xymatrix{
 & \mathcal{E}_\alpha \\
\mathcal{E}'_{\beta'} 
\ar[ru]^{\Psi_{\alpha \beta'}} \ar[r]_{\Psi_{\beta \beta'}} &
\mathcal{E}_\beta \ar[u]_{\Phi_{\alpha \beta}} 
}
$$
\end{enumerate}
\end{dfn}

\medskip

It may be helpful to give a more explicit description than that in the definition above. We can describe a twisted vectorial bundle as the data
$$
(\U, s_\alpha, g_{\alpha \beta}, (E_\alpha, h_\alpha), \phi_{\alpha \beta})
$$
consisting of:
\begin{list}{$\bullet$}{\topsep=4pt}
\item
an open cover $\U = \{ U_\alpha \}$ of $X$;

\item
local sections $s_\alpha : U_\alpha \to P|_{U_\alpha}$;

\item
lifts $g_{\alpha \beta} : U_{\alpha \beta} \to U(\H)$ of the transition functions $\bar{g}_{\alpha \beta}$;

\item
$\Z_2$-graded Hermitian vector bundles $E_\alpha \to U_\alpha$ of finite rank;

\item
Hermitian maps $h_\alpha : E_\alpha \to E_\alpha$ of degree 1;

\item
maps $\phi_{\alpha \beta} : E_\beta|_{U_{\alpha \beta}} \to E_\alpha|_{U_{\alpha \beta}}$ such that $h_\alpha \phi_{\alpha \beta} = \phi_{\alpha \beta} h_\beta$ and:
$$
\begin{array}{rcll}
\phi_{\alpha \beta} \phi_{\beta \alpha} & \fallingdotseq & 1 & 
\mbox{on $U_{\alpha \beta}$}; \\
\phi_{\alpha \beta} \phi_{\beta \gamma} & \fallingdotseq & 
z_{\alpha \beta \gamma} \phi_{\alpha \gamma} &
\mbox{on $U_{\alpha \beta \gamma}$}.
\end{array}
$$
\end{list}
In the above, the transition function $\bar{g}_{\alpha \beta} : U_{\alpha \beta} \to PU(\H)$ is defined by $s_\alpha \bar{g}_{\alpha \beta} = s_\beta$. A lift $g_{\alpha \beta}$ of $\bar{g}_{\alpha \beta}$ means a function $g_{\alpha \beta} : U_{\alpha \beta} \to U(\H)$ such that $\pi \circ g_{\alpha \beta} = \bar{g}_{\alpha \beta}$. The function $z_{\alpha \beta \gamma} : U_{\alpha \beta \gamma} \to U(1)$ is defined by $g_{\alpha \beta} g_{\beta \gamma} = z_{\alpha \beta \gamma} g_{\alpha \gamma}$.

\medskip

Note that the data $s_\alpha, g_{\alpha \beta}$ of $P$ are crucial in considering isomorphisms classes of twisted vectorial bundles.

\begin{dfn}
We denote by $\KF_P(X)$ the homotopy classes of twisted vectorial bundles over $X$ twisted by $P$.
\end{dfn}

The notion of \textit{homotopies} of $P$-twisted vectorial bundles over $X$ is formulated by using $(P \times I)$-twisted vectorial bundles over $X \times I$. As in the case of $\KF(X)$, the set $\KF_P(X)$ gives rise to an abelian group. Clearly, if $P$ is trivial, then a trivialization $P \cong X \times PU(\H)$ induces an isomorphism $KF_P(X) \cong KF(X)$.

\medskip

\begin{rem}
Consider the following property of a topological space $X$:
\begin{itemize}{}
\item[(L)]
For any principal $PU(\H)$-bundle $P \to X$ and an open cover of $X$, there is a refinement $\U = \{ U_\alpha \}$ of the cover such that we can find local trivializations $s_\alpha : U_\alpha \to P|_{U_\alpha}$ and lifts $g_{\alpha \beta}$ of the transition functions $\bar{g}_{\alpha \beta}$. 
\end{itemize}
As $P$ is locally trivial, the existence of lifts $g_{\alpha \beta}$ matters only. In general, paracompact spaces have the property (L). Thus, through this property, the paracompactness assumption in Definition \ref{dfn:twisted_vectorial_bundle} ensures that $\KF_P(X)$ is non-empty.
\end{rem}

\begin{rem}
The assignment of $\P(U)$ to each open set $U \subset X$ gives a \textit{$\underline{U(1)}$-gerbe} over $X$, where $\underline{U(1)}$ is the sheaf of germs of $U(1)$-valued functions. In general, for a $\underline{U(1)}$-gerbe $\mathcal{G}$, we can construct a category $\cKF_{\mathcal{G}}(X)$ similar to $\cKF_{P}(X)$. On a manifold $X$, the assignment $U \to \cKF_{\mathcal{G}}(U)$ becomes a \textit{stack} and gives the generalization of Brylinski's 2-vector bundle mentioned in Section \ref{sec:introduction}.
\end{rem}


\section{Cohomology theory $\KF$}
\label{sec:cohomology_KF}

By means of $\KF_P(X)$, we construct in this section a certain generalized cohomology theory similar to twisted $K$-cohomology theory. Then we describe and prove some basic properties.


\subsection{Construction}

Let $X$ be a paracompact space, and $P \to X$ a principal $PU(\H)$-bundle. We define the \textit{support} of a twisted vectorial bundle 
$$
\mathbb{E} = 
(\U, s_\alpha, g_{\alpha \beta}, (E_\alpha, h_\alpha), \phi_{\alpha \beta}) 
\in \cKF_P(X)
$$ 
to be:
$$
\mathrm{Supp} \mathbb{E} 
=
\overline{
\{ x \in X |\ \mbox{$(h_\alpha)_x$ is not invertible for some $\alpha$ } \}
}.
$$
For a (closed) subspace $Y \subset X$, we denote by $\cKF_P(X, Y)$ the full subcategory in $\cKF_P(X)$ consisting of objects $\mathbb{E}$ such that $\mathrm{Supp}\mathbb{E} \cap Y = \emptyset$. Then we define $\KF_P(X, Y)$ to be the homotopy classes of objects in $\cKF_P(X, Y)$, where homotopies are given by objects in $\cKF_{P \times I}(X \times I, Y \times I)$. For $n \ge 0$, we put:
$$
\KF_P^{-n}(X, Y) =
\KF_{P \times I^n}(X \times I^n, Y \times I^n \cup X \times \partial I^n).
$$
We also put $\KF_P^1(X, Y) = \KF_P^{-1}(X, Y)$. By means of the pull-back, a morphism $(f, F) : (X', Y'; P') \to (X, Y; P)$ in $\widehat{\cC}$ clearly induces a homomorphism $(f, F)^* : \KF_P^n(X, Y) \to \KF_{P'}^n(X', Y')$. In the case of $P' = f^*P$ and $F = \widehat{f}$, we will write $f^*$ for the induced homomorphism.

\begin{prop} \label{prop:cohomology_KF}
The assignment of $\{ \KF_P^n(X, Y) \}_{n \le 1}$ to $(X, Y; P) \in \widehat{\cC}$ has the following properties:
\begin{enumerate}
\item (Homotopy axiom)
If $(f_i, F_i) : (X', Y'; P') \to (X, Y; P)$, ($i = 0, 1$) are homotopic, then the induced homomorphisms coincide: $(f_0, F_0)^* = (f_1, F_1)^*$.

\item (Excision axiom)
For subcomplexs $A, B \subset X$, the inclusion map induces the isomorphism:
$$
\KF_{P|_{A \cup B}}^n(A \cup B, B) \cong 
\KF_{P|_{A}}^n(A, A \cap B). 
$$

\item (``Exactness'' axiom)
There is the natural complex of groups:
$$
\hspace{-0.5cm}
\cdots \to
\KF^{-1}_{P|_Y}(Y) \overset{\delta_{-1}}{\to}
\KF^{0}_P(X, Y) \to
\KF^{0}_P(X) \to
\KF^{0}_{P|_Y}(Y) \overset{\delta_{0}}{\to}
\KF^{1}_P(X, Y).
$$
This complex is exact except at the term $\KF^{0}_{P|_Y}(Y)$.

\item (Additivity axiom)
For a family $\{ (X_\lambda, Y_\lambda; P_\lambda) \}_{\lambda \in \Lambda}$ in $\widehat{\cC}$, the inclusion maps $X_\lambda \to \coprod_\lambda X_\lambda$ induce the natural isomorphism:
$$
\KF_{\coprod_\lambda P_\lambda}^{-n}
(\tcoprod_\lambda X_\lambda, \tcoprod_\lambda Y_\lambda)
\cong
\prod_\lambda 
\KF_{P_\lambda}^{-n}(X_\lambda, Y_\lambda). 
$$
\end{enumerate}
\end{prop}

The homotopy axiom and the additivity axiom follow directly from the definition of $\KF_P(X, Y)$. The excision axiom and the ``exactness'' axiom will be shown in the following subsections. 

\medskip

\begin{rem}
The Bott periodicity for $\KF_P^{-n}(X, Y)$ is not yet established at this stage. This is the reason that the ``exactness'' axiom in Proposition \ref{prop:cohomology_KF} is formulated partially. At the end, the periodicity will turn out to hold, and we will obtain the complete exactness axiom.
\end{rem}

\begin{rem}
A generalization of $\KF_P(X)$ is given by incorporating actions of Clifford-algebra bundles into vectorial bundles. Another generalization is to use real vector bundles with inner product instead of Hermitian vector bundles. These generalizations also satisfy properties similar to those in Proposition \ref{prop:cohomology_KF}.
\end{rem}


\subsection{Excision axiom}

We here prove the excision axiom in Proposition \ref{prop:cohomology_KF}. For untwisted $\KF(X, Y)$, the excision theorem is shown in \cite{F1}. The following argument is essentially the same as that used in the untwisted case.

\begin{lem}[Meyer-Vietoris construction] \label{lem:Meyer-Vietoris}
Let $X$ be a paracompact space, $P \to X$ a principal $PU(\H)$-bundle, and $U, V \subset X$ open subsets such that $U \cap V \neq \emptyset$. If $\mathbb{E} \in \cKF_{P|_U}(U)$ and $\mathbb{F} \in \cKF_{P|_V}(V)$ are isomorphic on $U \cap V$, then there is $\mathbb{G} \in \cKF_{P|_{U \cup V}}(U \cup V)$ such that $\mathbb{G}|_U \cong \mathbb{E}$ and $\mathbb{G}|_V \cong \mathbb{F}$.
\end{lem}

\begin{proof}
Suppose that $\mathbb{E} = (\mathcal{U}, \mathcal{E}_\alpha, \Phi_{\alpha \alpha'})$ and $\mathbb{F} = (\mathcal{V}, \mathcal{F}_\beta, \Psi_{\beta \beta'})$. We can construct the object $\mathbb{G} = ( \mathcal{W}, \mathcal{G}_\gamma, \Upsilon_{\gamma \gamma'})$ as follows. We let $\mathcal{W}$ be the open cover of $U \cup V$ consisting of the open sets belonging to $\mathcal{U}$ or $\mathcal{V}$. The object $\mathcal{G}_\gamma$ is $\mathcal{E}_\alpha$ or $\mathcal{F}_\beta$. Then $\Phi_{\alpha \alpha'}$, $\Psi_{\beta \beta'}$ and the data of the isomorphism $\mathbb{E}|_{U \cap V} \cong \mathbb{F}|_{U \cap V}$ together give the morphisms $\Upsilon_{\gamma \gamma'}$.
\end{proof}

\begin{prop}[Excision axiom]
Let $X$ be a paracompact space, and $P \to X$ a principal $PU(\H)$-bundle. For an open set $U$ and closed set $Y$ such that $U \subset Y \subset X$, the inclusion $i : X - U \to X$ induces the isomorphism:
$$
i^* : \ \KF_P(X, Y) 
\overset{\cong}{\longrightarrow} \KF_{P|_{X - U}}(X - U, Y - U).
$$
\end{prop}

\begin{proof}
It suffices to construct the inverse of $i^*$. Suppose that we are given $\mathbb{E} \in \cKF_{P|_{X - U}}(X - U, Y - U)$. We put $V = X - Y$ and $W = X - \mathrm{Supp}\mathbb{E}$. We let $\mathbb{O} = (\U, s_\alpha, g_{\alpha \beta}, (E_\alpha, h_\alpha), \phi_{\alpha \beta}) \in \cKF_{P|_W}(W)$ be an object such that $E_\alpha$, $h_\alpha$ and $\phi_{\alpha \beta}$ are trivial. Note that $\mathbb{O}$ represents $0 \in \KF_{P|_W}(W)$. Clearly, the support of $\mathbb{E}$ does not intersect $V \cap W$. Thus, there is a natural isomorphism $\mathbb{E}|_{V \cap W} \cong \mathbb{O}|_{V \cap W}$, so that Lemma \ref{lem:Meyer-Vietoris} provides us an object $\tilde{\mathbb{E}} \in \cKF_P(X)$ such that $\mathrm{Supp}\tilde{\mathbb{E}} \cap Y = \emptyset$ and $\tilde{\mathbb{E}}|_{X - U} \cong \mathbb{E}$. Note that the construction in Lemma \ref{lem:Meyer-Vietoris} is natural. Hence the construction of $\tilde{\mathbb{E}}$ above behaves naturally with respect to the pull-back. Consequently, the assignment $\mathbb{E} \mapsto \tilde{\mathbb{E}}$ induces a well-defined map $\KF_{P|_{X-U}}(X - U, Y - Y) \to \KF_P(X, Y)$, giving the inverse to $i^*$.
\end{proof}

Now, the excision axiom in Proposition \ref{prop:cohomology_KF} follows from the proposition above: Setting $X = A \cup B$, $Y = B$ and $U = B - A \cap B$, we get $\KF_P(A \cup B, B) \cong \KF_{P|_A}(A, A \cap B)$, which leads to $KF^n_P(A \cup B, B) \cong KF^n_{P|_A}(A, A \cap B)$, ($n < 0$).


\subsection{Exactness axiom}
\label{subsec:exactness_axiom}

We show the ``exactness'' axiom in Proposition \ref{prop:cohomology_KF} in a way similar to that used in \cite{A}. To define the connecting homomorphism $\delta_{-n}$, we begin with:

\begin{lem} \label{lem:triple}
Let $X$ be a paracompact space, $P \to X$ a principal $PU(\H)$-bundle, and $Z \subset Y \subset X$ subspaces. If $Y \to X$ is a cofibration, then we have the exact sequence:
$$
\KF_P(X, Y) \overset{i^*}{\longrightarrow} 
\KF_P(X, Z) \overset{j^*}{\longrightarrow} 
\KF_{P|_Y}(Y, Z),
$$
where $i : (X, Z) \to (X, Y)$ and $j : (Y, Z) \to (X, Y)$ are the inclusion maps.
\end{lem}

\begin{proof}
Clearly, $j^*i^* = 0$. Suppose that $[\mathbb{E}] \in \KF_P(X, Z)$ is such that $j^*([\mathbb{E}]) = 0$. This means that there is a homotopy $\tilde{\mathbb{F}} \in \cKF_{P \times I}(Y \times I, Z \times I)$ connecting $\mathbb{E}|_Y$ with a trivial object $\mathbb{O}$ on $Y$. Then we have an object 
$$
\mathbb{G} \in 
\cKF_{P \times [0, 1]}(X \cup \{ 0 \} \cup Y \times [0, 1], Z \times [0, 1])
$$ 
such that $\mathbb{G}|_{X \times \{ 0 \}} \cong \mathbb{E}$ and $\mathbb{G}|_{Y \times \{ 1 \}} \cong \mathbb{O}$. We can construct such an object $\mathbb{G}$ by applying Lemma \ref{lem:Meyer-Vietoris} to $\tilde{\mathbb{F}}|_{Y \times (0, 1]}$ and the pull-back of $\mathbb{E}$ under the projection $X \times \{ 0 \} \cup Y \times [0, 1) \to X \times \{ 0 \}$. Now, because $Y \to X$ is a cofibration, we have a map $\eta$ making the following diagram commutative.
$$
\xymatrix{
X \times \{ 0 \} \cup Y \times [0, 1] 
\ar@{^{(}->}[r] \ar[d]_{\mathrm{id}} &
X \times [0, 1] \ar@{-->}[ld]^-\eta \\
X \times \{ 0 \} \cup Y \times [0, 1] &
}
$$
Then $\mathbb{H} = \eta(\cdot, 1)^*\mathbb{G}$ defines $[\mathbb{H}] \in \KF_P(X, Y)$. Since the homotopy $\tilde{\mathbb{H}} = \eta^*\mathbb{G} \in \cKF_{P \times I}(X \times I, Z \times I)$ connects $\mathbb{G}|_{X \times \{ 0 \} } \cong \mathbb{E}$ with $\mathbb{H}$, we have $i^*([\mathbb{H}]) = [\mathbb{E}]$.
\end{proof}

\begin{lem} \label{lem:iso_difficult}
For $(X, Y; P) \in \widehat{\cC}$, the group $\KF_P^{-n+1}(X, Y)$ is isomorphic to 
\begin{multline*}
\KF_{P \times I^{n}}(
X \times I^{n-1} \times \{ 0 \} \cup 
(Y \times I^{n-1} \cup X \times \partial I^{n-1}) \times I, \\
Y \times I^{n-1} \times \{ 1 \} \cup
X \times \partial I^{n-1} \times I
).
\end{multline*}
\end{lem}

\begin{proof}
Recall $\KF_P^{-n+1}(X, Y) = \KF_{P \times I^{n-1}}(X \times I^{n-1}, Y \times I^{n-1} \cup X \times \partial I^{n-1})$ by definition. Consider the following maps:
$$
\xymatrix{
{
\begin{array}{c}
(X \times I^{n-1} \times \{ 0 \}
\cup (Y \times I^{n-1} \cup X \times \partial I^{n-1}) \times I,
\quad \quad \quad  \\
\quad \quad \quad \quad \quad \quad \quad \quad \quad \quad
Y \times I^{n-1} \times \{ 1 \} \cup X \times \partial I^{n-1} \times I)
\end{array}
} 
\ar@{_{(}->}[d]_i \\
{
\begin{array}{c}
(X \times I^{n-1} \times \{ 0 \}
\cup (Y \times I^{n-1} \cup X \times \partial I^{n-1}) \times I,
\quad \quad \quad  \\
\quad \quad \quad \quad \quad \quad \quad \quad \quad \quad
(Y \times I^{n-1} \cup X \times \partial I^{n-1}) \times I)
\end{array}
} 
\ar@{->}[d]_p \\
(X \times I^{n-1} \times \{ 0 \}, \quad
(Y \times I^{n-1} \cup X \times \partial I^{n-1}) \times \{ 0 \}),
\ar@<-3ex>@{^{(}->}[u]_j
}
$$
where $i$ and $j$ are induced from the inclusions, and $p$ from the projection. The map $p' = p \circ i$ is also induced from the projection. We will prove below that ${p'}^*$ provides us the isomorphism in the present lemma. For the aim, we show:
\begin{list}{}{\parsep=-2pt\topsep=4pt}
\item[(a)]
${p'}^*$ is injective;

\item[(b)]
$p^*$ is surjective; and

\item[(c)]
$i^*$ is surjective.
\end{list}
For (a), we use the map $H$ given by the homotopy extension property:
$$
\xymatrix{
X \times I^{n-1} \times \{ 0 \} \cup
(Y \times I^{n-1} \cup X \times \partial I^{n-1}) \times I \
\ar@<-0.5ex>@{^{(}->}[r] \ar[d]^{\mathrm{id}} &
X \times I^{n-1} 
\times I \ar@{-->}[dl]^{H} \\
X \times I^{n-1} \times \{ 0 \} \cup
(Y \times I^{n-1} \cup X \times \partial I^{n-1}) \times I.
}
$$
If we put $h(\cdot) = H(\cdot, 1)$, then $p' \circ h$ is homotopic to the identity of $X \times I^{n-1}$ relative to $Y \times I^{n-1} \cup X \times \partial I^{n-1}$, so that ${p'}^*$ is injective. For (b), it is enough to apply the excision axiom. For (c), we define $\mathcal{X} \supset \mathcal{Y} \supset \mathcal{Z}$ as follows.
\begin{align*}
\mathcal{X} &= 
X \times I^{n-1} \times \{ 0 \} 
\cup (Y \times I^{n-1} \cup X \times \partial I^{n-1}) \times I, \\
\mathcal{Y} &=
(Y \times I^{n-1} \cup X \times \partial I^{n-1}) \times I, \\
\mathcal{Z} &=
Y \times I^{n-1} \times \{ 1 \} \cup X \times \partial I^{n-1} \times I.
\end{align*}
Lemma \ref{lem:triple} gives the exact sequence:
$$
\begin{CD}
\KF_{P \times I^n}(\mathcal{X}, \mathcal{Y}) @>>>
\KF_{P \times I^n}(\mathcal{X}, \mathcal{Z}) @>>>
\KF_{P \times I^n}(\mathcal{Y}, \mathcal{Z}),
\end{CD}
$$
in which the first map coincides with $i^*$. Hence the surjectivity of $i^*$ will follow from $\KF_{P \times I^n}(\mathcal{Y}, \mathcal{Z}) = 0$. To see this vanishing, we let $(\mathcal{Z}, \mathcal{Z}) \to (\mathcal{Y}, \mathcal{Z})$ be the inclusion, and $(\mathcal{Y}, \mathcal{Z}) \to (\mathcal{Z}, \mathcal{Z})$ the map given by composing the following projection and inclusion:
$$
\begin{CD}
\mathcal{Y} @>>>
(Y \times I^{n-1} \cup X \times \partial I^{n-1}) \times \{ 1 \} @>>>
\mathcal{Z}.
\end{CD}
$$
These maps give homotopy equivalences between $(\mathcal{Y}, \mathcal{Z})$ and $(\mathcal{Z}, \mathcal{Z})$, so that we have $\KF_{P \times I^n}(\mathcal{Y}, \mathcal{Z}) \cong \KF_{P \times I^n}(\mathcal{Z}, \mathcal{Z}) = 0$.
\end{proof}

\begin{lem} \label{lem:iso_easy}
For $(X, Y; P) \in \widehat{\cC}$, the group $\KF_{P|_Y}^{-n}(Y)$ is isomorphic to
\begin{multline*}
\KF_{P \times I^{n}}(
X \times I^{n-1} \times \{ 0 \} \cup 
(Y \times I^{n-1} \cup X \times \partial I^{n-1}) \times I, \\
X \times I^{n-1} \times \{ 0 \} \cup 
Y \times I^{n-1} \times \{ 1 \} \cup
X \times \partial I^{n-1} \times I
).
\end{multline*}
\end{lem}

\begin{proof}
The present lemma straightly follows from the excision axiom.
\end{proof}

Now, for $(X, Y; P) \in \widehat{\cC}$, we define the natural homomorphism
$$
\delta_{-n} : \ 
\KF_{P|_Y}^{-n}(Y) \longrightarrow 
\KF_P^{-n+1}(X, Y), \quad (n \ge 1) 
$$
to be the composition of the isomorphism in Lemma \ref{lem:iso_easy}, the following homomorphism induced from the inclusion map:
$$
\begin{CD}
\begin{array}{c}
\KF_{P \times I^{n}}(
X \times I^{n-1} \times \{ 0 \} \cup 
(Y \times I^{n-1} \cup X \times \partial I^{n-1}) \times I, 
\hspace{2.5cm} \\
\hspace{3cm}
X \times I^{n-1} \times \{ 0 \} \cup 
Y \times I^{n-1} \times \{ 1 \} \cup
X \times \partial I^{n-1} \times I
)
\end{array} \\
@VVV \\
\begin{array}{c}
\KF_{P \times I^{n}}(
X \times I^{n-1} \times \{ 0 \} \cup 
(Y \times I^{n-1} \cup X \times \partial I^{n-1}) \times I, 
\hspace{2.5cm} \\
\hspace{5cm}
Y \times I^{n-1} \times \{ 1 \} \cup
X \times \partial I^{n-1} \times I
),
\end{array}
\end{CD}
$$
and the isomorphism in Lemma \ref{lem:iso_difficult}.

\begin{prop}[Exactness axiom] \label{prop:exactness}
For $(X, Y; P) \in \widehat{\cC}$ and $n \ge 0$, we have the following exact sequences.
\begin{enumerate}
\item[(a)] \
$
\KF_P^{-n}(X, Y) \longrightarrow 
\KF_P^{-n}(X) \longrightarrow 
\KF_{P|_Y}^{-n}(Y).
$

\item[(b)] \
$
\KF_{P|_Y}^{-n-1}(Y) \overset{\delta_{-n-1}}{\longrightarrow} 
\KF_P^{-n}(X, Y) \longrightarrow 
\KF_P^{-n}(X).
$

\item[(c)] \
$
\KF_P^{-n-1}(X) \longrightarrow 
\KF_{P|_Y}^{-n-1}(Y) \overset{\delta_{-n-1}}{\longrightarrow} 
\KF_P^{-n}(X, Y).
$
\end{enumerate}
In the above, the maps $\KF_P^{-n}(X, Y) \to \KF_P^{-n}(X)$ and $\KF_P^{-n}(X) \to \KF_{P|_Y}^{-n}(Y)$ are induced from the inclusions $(X, \emptyset) \to (X, Y)$ and $Y \to X$, respectively.
\end{prop}

\begin{proof}
We define $\mathcal{X} \supset \mathcal{Y} \supset \mathcal{Z}$ to be $\mathcal{X} = X \times I^n$, $\mathcal{Y} = Y \times I^n \cup X \times \partial I^n$ and $\mathcal{Z} = X \times \partial I^n$. Then we consider the diagram:
$$
\begin{CD}
\KF_{P \times I^n}(\mathcal{X}, \mathcal{Y}) @>>>
\KF_{P \times I^n}(\mathcal{X}, \mathcal{Z}) @>>>
\KF_{P \times I^n}(\mathcal{Y}, \mathcal{Z}) \\
@| @| @VVV \\
\KF_P^{-n}(X, Y) @>>>
\KF_P^{-n}(X) @>>>
\KF_{P|_Y}^{-n}(Y),
\end{CD}
$$
where the upper row is the exact sequence in Lemma \ref{lem:triple}, the lower row is the sequence in (a), and the third vertical map is the isomorphism in the excision axiom. The diagram above commutes, so that (a) is proved. We can prove (b) and (c) in the same way. For (b), we use $\mathcal{X} \supset \mathcal{Y} \supset \mathcal{Z}$ given by:
\begin{align*}
\mathcal{X} 
&= 
X \times I^{n} \times \{ 0 \} 
\cup (Y \times I^{n} \cup X \times I^{n}) \times I, \\
\mathcal{Y} 
&=
X \times I^{n} \times \{ 0 \} \cup \mathcal{Z}, \\
\mathcal{Z} 
&=
Y \times I^{n} \times \{ 1 \} \cup X \times \partial I^{n} \times I.
\end{align*}
For (c), we use $\mathcal{X} \supset \mathcal{Y} \supset \mathcal{Z}$ given by
\begin{align*}
\mathcal{X} 
&= 
X \times I^n \times [-1, 0] \cup 
X \times I^n \times \{ 0 \} \cup 
Y \times I^n \times [0, 1], \\
\mathcal{Y} 
&=
X \times I^n \times \{ -1 \} \cup
X \times I^n \times \{ 0 \} \cup
Y \times I^n \times [0, 1] \cup
X \times \partial I^n \times [-1, 0], \\
\mathcal{Z} 
&=
X \times I^n \times \{ -1 \} \cup
Y \times I^n \times \{ 1 \} \cup
Y \times \partial I^n \times [0, 1] \cup
X \times \partial I^n \times [-1, 0]. 
\end{align*}
Then we obtain (b) and (c) by identifying the groups in the exact sequence in Lemma \ref{lem:triple}, and checking the compatibility of the identifications. The details of the check are left to the reader (cf. \cite{A}).
\end{proof}

Finally, to complete the proof of the ``exactness'' axiom in Proposition \ref{prop:cohomology_KF}, we extend the exact sequence obtained so far in Proposition \ref{prop:exactness} as a complex. For the purpose, we let $F = F^0 \oplus F^1$ be the $\Z_2$-graded Hermitian vector bundle over the unit disk $D^2 \subset \C$ defined by $F^i = D^2 \times \C$. We also let ${T} : F \to F$ be the Hermitian map of degree $1$ given by ${T}_z = \left( \begin{array}{cc} 0 & \bar{z} \\ z & 0 \end{array} \right)$. As is known \cite{A}, the ``Thom class'' $(F, {T})$ represents a generator of $K(D^2, S^1) \cong \Z$.

\smallskip

For a moment, let $X$ be a paracompact space and $P \to X$ a principal $PU(\H)$-bundle. Multiplying $\mathbb{E} = (\{ U_\alpha \}_{\alpha \in \mathfrak{A}}, s_\alpha, g_{\alpha \beta}, (E_\alpha, h_\alpha), \phi_{\alpha \beta}) \in \cKF_P(X)$ and $(F, {T})$, we get the following object $\beta(\mathbb{E})$ in $\cKF_{P \times D^2}(X \times D^2, X \times S^1)$:
$$
(\{ U_\alpha \times D^2 \}_{\alpha \in \mathfrak{A}}, \
\pi_X^*s_\alpha, \ \pi_X^*g_{\alpha \beta}, \
(\pi_X^*E_\alpha {\otimes} \pi_{D^2}^*F, 
\pi_X^*h_\alpha \hat{\otimes} \pi_{D^2}^*{T}), \
\pi_X^*\phi_{\alpha \beta}),
$$
where $\pi_X$ and $\pi_{D^2}$ are the projections from $X \times D^2$ to $X$ and $D^2$ respectively. The Hermitian map $\pi_X^*h_\alpha \hat{\otimes} \pi_{D^2}^*{T}$ of degree $1$,  acting on the $\Z_2$-graded tensor product $\pi_X^*E_\alpha {\otimes} \pi_{D^2}^*F$, is given by:
$$
\pi_X^*h_\alpha \hat{\otimes} \pi_{D^2}^*{T}
=
\pi_X^*h_\alpha \otimes 1 + \epsilon \otimes \pi^*_{D^2}T,
$$
where $\epsilon$ is $1$ on the even part of $\pi_X^*E_\alpha$, and $-1$ on the odd part. The assignment $\mathbb{E} \mapsto \beta(\mathbb{E})$ gives rise to a functor, and induces a natural homomorphism:
$$
\beta : \ \KF_P(X) \longrightarrow \KF_P^{-2}(X)
= \KF_{P \times D^2}(X \times D^2, Y \times S^1).
$$

Now, we complete the proof of the ``exactness'' axiom in Proposition \ref{prop:cohomology_KF}:

\begin{prop}
For $(X, Y; P) \in \widehat{\cC}$, we define $\delta_0 : \KF_{P|_Y}^0(Y) \to \KF_P^1(X, Y)$ to be $\delta_0 = \delta_{-2} \circ \beta$. Then the following maps compose to give the trivial map, i.e. $\delta_0 \circ i^* = 0$.
$$
\begin{CD}
\KF_P^0(X) @>i^*>>  
\KF_{P|_Y}^0(Y) @>\delta_0>> 
\KF_{P}^1(X, Y).
\end{CD}
$$
\end{prop}

\begin{proof}
Notice the commutative diagram:
$$
\begin{CD}
\KF^0_P(X) @>{i^*}>> \KF^0_{P|_Y}(Y) @>{\delta_0}>> \KF_P^1(X, Y)  \\
@V{\beta}VV @VV{\beta}V @| \\
\KF^{-2}_P(X) @>{i^*}>> \KF^{-2}_{P|_Y}(Y) @>{\delta_{-2}}>> \KF_P^{-1}(X,Y).
\end{CD}
$$
Now, Proposition \ref{prop:exactness} (c) completes the proof.
\end{proof}


\section{Finite-dimensional approximation}
\label{sec:fda}

In this section, we construct a natural transformation between $K^*_P(X, Y)$ and $\KF^*_P(X, Y)$. The key to the construction is a notion of a finite-dimensional approximation of a family of Fredholm operators. We then study some properties of the natural transformation to prove our main theorem (Theorem \ref{thm:main}).


\subsection{Approximation of family of Fredholm operators}
\label{subsec:fda}

First of all, we introduce some notations: let $\hat{\H}$ be the $\Z_2$-graded Hilbert space $\hat{\H} = \H \oplus \H$, and $\F(\hat{\H})$ the space of self-adjoint bounded operators on $\hat{\H}$ of degree 1 whose square differ from the identity by compact operators:
$$
\F(\hat{\H}) = \{ \hat{A} : \hat{\H} \to \hat{\H} |\ 
\mbox{bounded, self-adjoint, degree $1$, $\hat{A}^2 - 1 \in \K(\hat{\H})$} \}.
$$
We identify $\F(\H)$ with $\F(\hat{\H})$ through the assignment $A \mapsto \hat{A} = \left( \begin{array}{cc} 0 & A^* \\ A & 0 \end{array} \right)$. 

For $A \in \F(\H)$, we write $\rho(\hat{A}^2)$ for the resolvent set of the operator $\hat{A}^2$, and $\sigma(\hat{A}^2) = \C - \rho(\hat{A}^2)$ for the spectrum set. If $\mu$ is such that $0 < \mu < 1$, then $\sigma(\hat{A}^2) \cap [0, \mu)$ consists of a finite number of eigenvalues, since $\hat{A}^2 - 1$ is compact. In particular, corresponding eigenspaces are finite-dimensional, and so is the following direct sum:
$$
(\hat{\H}, \hat{A})_{< \mu} = 
\bigoplus_{\lambda < \mu} \Ker(\hat{A}^2 - \lambda) =
\bigoplus_{\lambda < \mu} 
\{ \xi \in \hat{\H} |\ \hat{A}^2 \xi = \lambda \xi \}.
$$

The purpose of this subsection is to establish:

\begin{prop} \label{prop:fda}
Let $X$ be a topological space, and $A : X \to \F(\H)$ a continuous map. For an open set $U \subset X$ and a number $\mu \in (0, 1) \cap \bigcap_{x \in U} \rho(\hat{A}^2_x)$ given, the family of vector spaces $\bigcup_{x \in U} (\hat{\H}, \hat{A}_x)_{< \mu} \subset U \times \hat{\H}$ gives rise to a (finite rank, $\Z_2$-graded, Hermitian) vector bundle over $U$.
\end{prop}

The restriction of $\hat{A}_x$ to $(\hat{\H}, \hat{A}_x)_{< \mu}$ approximates the original family $\{ \hat{A}_x \}$.

\medskip

Notice that the next lemma ensures the hypothesis in the proposition:

\begin{lem} \label{lem:ensure_hypothesis}
Let $A : X \to \F(\H)$ be a continuous map. For each point $x_0 \in X$ and a number $\mu \in \rho(\hat{A}^2_{x_0})$, there is an open neighborhood $U$ of $x_0$ such that:
$$
\mu \in \bigcap_{x \in U} \rho(\hat{A}^2_x).
$$
\end{lem}

\begin{proof}
This lemma follows from the following facts: (i) the map $X \to \B(\hat{\H})_\mathrm{norm}$, $(x \mapsto \hat{A}_x^2 - \mu)$ is continuous; (ii) the operator $\hat{A}^2_{x_0} - \mu$ is invertible; and (iii) invertible bounded operators on $\hat{\H}$ form an open subset in $\B(\hat{\H})_\mathrm{norm}$.
\end{proof}

For the proof of Proposition \ref{prop:fda}, we show some lemmas.

\begin{lem} \label{lem:dimension}
Let $U$ and $\mu$ be as in Proposition \ref{prop:fda}. For each point $x_0 \in U$, there exists an open neighborhood $V \subset U$ of $x_0$ such that:
$$
\quad \quad \quad
\mathrm{dim} (\hat{\H}, \hat{A}_x)_{< \mu} = 
\mathrm{dim} (\hat{\H}, \hat{A}_{x_0})_{< \mu} 
\quad < +\infty
$$
for all $x \in V$.
\end{lem}

\begin{proof}
Because $\hat{A}_{x_0}^2 - 1$ is compact, $(\hat{\H}, \hat{A}_{x_0})_{< \mu}$ is finite-dimensional. We put $r = \mathrm{dim} (\hat{\H}, \hat{A}_{x_0})_{< \mu}$. Let $\lambda_1(x) \le \lambda_2(x) \le \cdots $ denote eigenvalues of $\hat{A}_x^2$, where each eigenvalue is included as many times as the dimension of its eigenspace. As is known, $\lambda_k(x)$ is a continuous function in $x$, because of the expression:
$$
\lambda_k(x) = 
\sup_{\substack{E \subset \H \\ \mathrm{dim}E = k -1}} 
\inf_{u \in E^\perp - \{ 0 \}}
\frac{\langle u, \hat{A}_x^2 u \rangle}{\lvert u \rvert^2}.
$$
We choose $\varepsilon$ so as to be $0 < 2\varepsilon < \mathrm{min}\{ \mu - \lambda_r(x_0), \ \lambda_{r+1}(x_0) - \mu \}$, and define the open set $V$ such that $x_0 \in V \subset U$ to be:
$$
V = 
\bigcap_{i = 1}^{r+1}
\{ x \in U |\ \lvert \lambda_i(x) - \lambda_i(x_0) \rvert < \varepsilon \}.
$$
For all $x \in V$, we have $\lambda_r(x) < \mu < \lambda_{r+1}(x)$, so that $\mathrm{dim}(\hat{\H}, \hat{A}_{x})_{< \mu} = r$.
\end{proof}

\begin{lem} \label{lem:continuity}
Let $U$ and $\mu$ be as in Proposition \ref{prop:fda}. The orthogonal projections $\pi_x : \hat{\H} \to \hat{\H}$ onto $(\hat{\H}, \hat{A}^2_x)_{< \mu}$ constitute the continuous map 
$$
\pi = \{ \pi_x \}_{x \in U} : \ 
U \longrightarrow \mathcal{B}(\H)_\mathrm{norm}.
$$
\end{lem}

\begin{proof}
It suffices to prove that, for a point $x_0 \in U$, we have $\lVert \pi_x - \pi_{x_0} \rVert \to 0$ as $x \to x_0$. For this aim, we choose $\varepsilon$ and $V$ as in the proof of Lemma \ref{lem:dimension}. If $x \in V$, then $\pi_x$ has the expression:
$$
\pi_x = 
\frac{1}{2\pi i}\int_C R(z; \hat{A}_x^2) dz,
$$
where $ R(z; \hat{A}_x^2) = (z - \hat{A}_x^2)^{-1}$ is the resolvent, and $C$ is a counterclockwisely oriented circle in $\C$ such that: its center lies on the real axis; and the open disks $B(\lambda_i(x_0); 2\varepsilon)$, ($i = 1, \ldots, r$) are inside $C$, but $B(\lambda_{r+1}(x_0); 2\varepsilon)$ is outside. Notice $\varepsilon < \lvert z - \lambda \rvert$ for $(z, \lambda) \in C \times \bigcup_{x \in V} \sigma(\hat{A}^2_x)$. Thus, for $(z, x) \in C \times V$, we have:
$$
\lVert R(z; \hat{A}_x^2) \rVert 
=
\sup_{u \neq 0} \frac{\lVert R(z; \hat{A}_x^2) u \rVert}{\lVert u \rVert}
=
\sup_{v \neq 0} \frac{\lVert v \rVert}{\lVert (z - \hat{A}_x^2) v \rVert}
=
\frac{1}
{\inf_{v \neq 0} \frac{\lVert (z - \hat{A}_x^2) v \rVert}{\lVert v \rVert}}
<
\frac{1}{\varepsilon}.
$$
Now, thanks to the integral expression of $\pi_x$, we get a constant $M$ such that $\lVert \pi_x - \pi_{x_0} \rVert \le M \lVert \hat{A}_x^2 - \hat{A}_{x_0}^2 \rVert$ for $x \in V$. Hence $\lVert \pi_x - \pi_{x_0} \rVert \to 0$ as $x \to x_0$. 
\end{proof}

\begin{lem} \label{lem:trivialization}
Let $U$ and $\mu$ be as in Proposition \ref{prop:fda}. For $x_0 \in U$, there is an open neighborhood $W \subset U$ of $x_0$ such that: the projection $p : \hat{\H} \to (\hat{\H}, \hat{A}_{x_0})_{< \mu}$ induces an isomorphism $(\hat{\H}, \hat{A}_x)_{< \mu} \cong (\hat{\H}, \hat{A}_{x_0})_{< \mu}$ for all $x \in W$.
\end{lem}

\begin{proof}
Let $F^\perp$ be the orthogonal complement of $F = (\hat{\H}, \hat{A}_{x_0})_{< \mu}$. We write $p^\perp : \hat{\H} \to F^\perp$ for the projection, and $i^\perp : F^\perp \to \hat{\H}$ for the inclusion. The operator $\pi_x^\perp = 1 - \pi_x$ is apparently Fredholm, and $x \mapsto \pi_x^\perp$ is norm continuous by Lemma \ref{lem:continuity}. Thus, in the same way as that used in the appendix of \cite{A}, we can find an open neighborhood $V'$ of $x_0$ such that: the map $p^\perp \pi_x^\perp i^\perp : \ F^\perp \to F^\perp$ is bijective for all $x \in V'$. Now, by the map of exact sequences:
$$
\begin{CD}
0 @>>> F^\perp @>{i^\perp}>> \hat{\H} @>{p}>> F @>>> 0 \\
@. @V{p^\perp \pi_x^\perp i^\perp}VV @VV{p^\perp \pi_x^\perp}V @VVV @. \\
0 @>>> F^\perp @= F^\perp @>>> 0 @>>> 0,
\end{CD}
$$
we see that $p$ induces an isomorphism $\Ker p^\perp \pi_x^\perp \cong F$ for $x \in V'$. Note that $\Ker p^\perp \pi_x^\perp \supset \Ker \pi_x^\perp = (\hat{\H}, \hat{A}_x)_{< \mu}$. By Lemma \ref{lem:dimension}, the dimension of $(\hat{\H}, \hat{A}_x)_{< \mu}$ is equal to that of $F = (\hat{\H}, \hat{A}_{x_0})_{< \mu}$, provided that $x \in V$. Thus, $p$ induces an isomorphism $(\hat{\H}, \hat{A}_x)_{< \mu} \cong (\hat{\H}, \hat{A}_{x_0})_{< \mu}$ for all $x \in W = V \cap V'$.
\end{proof}

\begin{proof}[Proof of Proposition \ref{prop:fda}]
It suffices to see that the family $\bigcup_{x \in U} (\hat{\H}, \hat{A}_x)_{< \mu}$ is locally trivial. We consider the open neighborhood $W$ of a point $x_0 \in U$ in Lemma \ref{lem:trivialization}. Then, on $W$, the map $\mathrm{id} \times p : W \times \H \to W \times (\hat{\H}, \hat{A}_{x_0})_{< \mu}$ induces a local trivialization $\bigcup_{x \in W} (\hat{\H}, \hat{A}_x)_{< \mu} \to W \times (\hat{\H}, \hat{A}_{x_0})_{< \mu}$.
\end{proof}

\begin{rem}
Instead of $\F(\H)$, we can use the space of Fredholm operators with the norm topology to obtain the same claim as Proposition \ref{prop:fda}. A key to this case is that $0$ is a discrete spectrum of a non-invertible Fredholm operator.
\end{rem}


\subsection{Natural transformation}

Let $X$ be a paracompact space, and $P \to X$ a principal $PU(\H)$-bundle. We construct a natural homomorphism:
$$
\alpha : \ 
K_P(X) \longrightarrow \KF_P(X)
$$
as follows: suppose that a section $\mathbb{A} \in \Gamma(X, P \times_{Ad} \F(\H))$ is given. We choose an open cover $\U = \{ U_\alpha \}_{x \in \mathfrak{A}}$ such that there are local sections $s_\alpha : U_\alpha \to P|_{U_\alpha}$ and lifts of transition functions $g_{\alpha \beta} : U_{\alpha \beta} \to U(\H)$. The local sections of $P$ allow us to identify $\mathbb{A}$ with a collection of maps $A_\alpha : U_\alpha \to \F(\H)$ such that $A_\alpha = g_{\alpha \beta} A_\beta g_{\alpha \beta}^{-1}$ on $U_{\alpha \beta}$. Because of Lemma \ref{lem:ensure_hypothesis}, taking a refinement of $\U$ if necessary, we can find a positive number $\mu_\alpha$ such that $\mu_\alpha \in \bigcap_{x \in U_\alpha} \rho((\hat{A}_\alpha^2)_x)$. By Proposition \ref{prop:fda}, we get a finite rank $\Z_2$-graded Hermitian vector bundle $E_\alpha = \bigcup_{x \in U_\alpha}(\hat{\H}, (\hat{A}_\alpha)_x)_{< \mu_\alpha}$ over $U_\alpha$. The restriction of $\hat{A}_\alpha$ to $E_\alpha$ defines a Hermitian map $h_\alpha : E_\alpha \to E_\alpha$ of degree $1$. On $U_{\alpha \beta}$, we define $\phi_{\alpha \beta} : E_\beta \to E_\alpha$ to be the composition of the maps:
$$
\bigcup_{x \in U_{\alpha \beta}} 
(\hat{\H}, (\hat{A}_\beta)_x)_{< \mu_\beta} 
\to
U_{\alpha \beta} \times \hat{\H} 
\overset{\mathrm{id} \times g_{\alpha\beta}}{\longrightarrow}
U_{\alpha \beta} \times \hat{\H} 
\to
\bigcup_{x \in U_{\alpha \beta}} 
(\hat{\H}, (\hat{A}_\alpha)_x)_{< \mu_\alpha},
$$
where the first and third maps are the inclusion and projection, respectively. The data $\mathbb{E} = (\U, s_\alpha, g_{\alpha \beta}, (E_\alpha, h_\alpha), \phi_{\alpha \beta})$ is a $P$-twisted vectorial bundle over $X$. The isomorphism class of $\mathbb{E}$ is independent of the choice of $\mu_\alpha$, $s_\alpha$, $g_{\alpha \beta}$ and $\mathcal{U}$. Now, the homomorphism $\alpha : K_P(X) \to \KF_P(X)$ is given by $\alpha([\mathbb{A}]) = [\mathbb{E}]$.

\medskip

The same construction yields a natural map $\alpha : K_P(X, Y) \to \KF_P(X, Y)$ for $Y \subset X$, and hence $\alpha_n : K^n_P(X, Y) \to \KF_P^n(X, Y)$. 

\begin{lem} \label{lem:compatible_with_periodicity}
For a paracompact space $X$ and a principal $PU(\H)$-bundle $P \to X$, the following diagram commutes:
$$
\begin{CD}
K_P(X) @>\beta>> K_{P \times D^2}(X \times D^2, X \times S^1) \\
@V{\alpha}VV  @VV{\alpha}V \\
\KF_P(X) @>{\beta}>> \KF_{P \times D^2}(X \times D^2, X \times S^1),
\end{CD}
$$
where the upper map $\beta$ induces the Bott periodicity $K_P^0(X) \cong K_P^{-2}(X)$.
\end{lem}

Before the proof of this lemma, we explain the map $\beta$ inducing the Bott periodicity for twisted $K$-cohomology. Roughly, the map is a ``multiplication of a Thom class''. To be more precise, recall the identification of $\F(\H)$ with $\F(\hat{\H})$. This identification is compatible with the conjugate actions of $PU(\H)$, through the diagonal embedding $U(\H) \to U(\hat{\H})$. So we can represent an element in $K_P(X)$ by a section $\hat{\mathbb{A}}$ of the associated bundle $P \times_{Ad} \F(\hat{\H})$ over $X$. We identify the section with a map $\hat{\mathbb{A}} : P \to \F(\hat{\H})$. For $(p, z) \in P \times D^2$, we define a degree $1$ self-adjoint Fredholm operator $\beta(\mathbb{A})_{(p, z)}$ on the $\Z_2$-graded Hilbert space $\hat{\H} \otimes (\C \oplus \C)$ by
\begin{align*}
\beta(\hat{\mathbb{A}})_{(p, z)} 
&= 
\hat{\mathbb{A}}_p \otimes 1 + \epsilon \otimes {T}_z \\
&= 
\left(
\begin{array}{cc}
0 & \mathbb{A}_p^* \\
\mathbb{A}_p & 0
\end{array}
\right)
\otimes
\left(
\begin{array}{cc}
1 & 0 \\
0 & 1
\end{array}
\right)
+
\left(
\begin{array}{cc}
1 & 0 \\
0 & -1
\end{array}
\right)
\otimes
\left(
\begin{array}{cc}
0 & \bar{z} \\
z & 0
\end{array}
\right).
\end{align*}
This operator defines a section $\beta(\hat{\mathbb{A}})$ of $(P \times D^2) \times_{Ad} \F(\hat{\H} \otimes (\C \oplus \C))$ over $X \times D^2$, and induces the Bott periodicity map $\beta : K_P(X) \to K_{P \times D^2}(X \times D^2, X \times S^1)$.

\begin{proof}[Proof of Lemma \ref{lem:compatible_with_periodicity}]
Let $\hat{\mathbb{A}} \in \Gamma(X, P \times_{Ad} \F(\hat{\H}))$ be a section given. First, we describe $\beta \alpha ([\hat{\mathbb{A}}]) \in \KF^{-2}_P(X)$ as follows: let an object
$$
\mathbb{E} = 
( \{ U_\alpha \}_{\alpha \in \mathfrak{A}}, 
s_\alpha, g_{\alpha \beta}, 
(E_\alpha, h_\alpha), \phi_{\alpha \beta}) 
\in \cKF_P(X)
$$ 
represent the element $\alpha([\hat{\mathbb{A}}]) \in \KF_P(X)$. We suppose that $E_\alpha$ is given by $E_\alpha = \bigcup_{x \in U_\alpha} (\hat{\H}, (\hat{A}_\alpha)_x)_{< \mu_\alpha}$ under a choice of a positive number $\mu_\alpha$. Taking a finer open cover if necessary, we can assume that the rank of $E_\alpha$ is $r_\alpha$, and that there exists a positive number $\epsilon_\alpha$ such that $\epsilon_\alpha < \mathrm{min}\{ \mu_\alpha - \lambda_{r_\alpha}(x), \ \lambda_{r_\alpha+1}(x) - \mu_\alpha \}$ for all $x \in U_\alpha$. Here $\lambda_j(x)$ is the $j$th eigenvalue of $(\hat{A}_\alpha^2)_x$, which varies continuously in $x$. (cf.\ Lemma \ref{lem:dimension}) We define an open cover $\{ V(s; \epsilon_\alpha) \}_{s \in [0, 1]}$ of $D^2$ by setting
$$
V(s; \epsilon_\alpha) =
\{ z \in D^2 |\ 
s - \epsilon_\alpha < \lvert z \rvert^2 < s + \epsilon_\alpha \}, \quad
s \in [0, 1].
$$
Then we can represent $\beta \alpha ([\hat{\mathbb{A}}]) \in \KF^{-2}_P(X)$ by
$$
(\{ U_\alpha \times V(s; \epsilon_\alpha) \}, 
\pi_X^*s_\alpha, \ \pi_X^*g_{\alpha \beta}, \
(\pi_X^*E_\alpha \otimes (\C \oplus \C), 
\pi_X^*h_\alpha \hat{\otimes} \pi_{D^2}^*{T}), \
\pi_X^*\phi_{\alpha \beta}).
$$
Next, we consider the element $\alpha \beta ([\hat{\mathbb{A}}]) \in \KF^{-2}_P(X)$. In applying Proposition \ref{prop:fda} to the section $\beta(\hat{\mathbb{A}})$ of $(P \times D^2) \times_{Ad} \F(\hat{\H} \otimes (\C \oplus \C))$, we use the open set $U_\alpha \times V(s; \epsilon_\alpha)$ and the positive number $\mu_\alpha + s$. The $j$th eigenvalue of the square of $(\hat{A}_\alpha)_x \otimes 1 + \epsilon \otimes {T}_z$ is $\lambda_j(x) + \lvert z \rvert^2$. Since $\lambda_{r_\alpha}(x) + \lvert z \rvert^2 < \mu_\alpha + s < \lambda_{r_\alpha+1}(x) + \lvert z \rvert^2$ holds for $(x, z) \in U_\alpha \times V(s; \epsilon_\alpha)$ by construction, we obtain:
$$
(\hat{\H} \otimes (\C \oplus \C), 
(\hat{A}_\alpha)_x \otimes 1 + \epsilon \otimes {T}_z)_{< \mu_\alpha + s}
\cong
(\hat{\H}, (\hat{A}_\alpha)_x)_{< \mu}
\otimes
(\C \oplus \C).
$$
Hence the representative of $\beta\alpha([\hat{\mathbb{A}}])$ also represents $\alpha \beta ([\hat{\mathbb{A}}])$.
\end{proof}

\begin{prop} \label{prop:natural_transformation}
The homomorphisms $\alpha_n : K_P^n(X, Y) \to \KF_P^n(X, Y)$, $(n \le 1)$ constitute a natural transformation of cohomology theories. 
\end{prop}

\begin{proof}
It suffices to see that the natural homomorphisms $\alpha_n$ are compatible with the axioms in Proposition \ref{prop:twisted_K_cohomology} and \ref{prop:cohomology_KF}. The homotopy axioms, the excision axioms and the additivity axioms are clearly compatible with $\alpha_n$. For $n < 0$, inclusion maps define $\delta_n$, so that $\delta_n \alpha_n = \alpha_{n+1} \delta_n$. This formula also holds for $n = 0$, because of Lemma \ref{lem:compatible_with_periodicity}. As a result, the ``exactness'' axioms are compatible with $\alpha_n$.
\end{proof}


\subsection{Finite-dimensional approximation in untwisted case}
\label{subsec:property_of_fda}

In untwisted case, the map $\alpha$ has the following property:

\begin{prop} \label{prop:fda_untwisted}
If $(X, Y; P) \in \widehat{\cC}$ is such that $P \to X$ is trivial, then the homomorphism $\alpha : K_P(X, Y) \to \KF_P(X, Y)$ is bijective.
\end{prop}

If $P$ is trivial, then we can identify $K_P(X, Y)$ with the set of the homotopy classes of maps $A : X \to \F(\H)$ such that $A_y$, ($y \in Y$) is the identify. Accordingly, we identify the map in Proposition \ref{prop:fda_untwisted} with:
$$
\alpha :  [(X, Y), (\F(\H), 1)] \longrightarrow \KF(X, Y),
$$
where the group $\KF(X, Y)$ consists of homotopy classes of vectorial bundles whose supports do not intersect $Y$. 

\medskip

For the proof of Proposition \ref{prop:fda_untwisted}, we notice Furuta's result:

\begin{thm}[\cite{F1}] \label{thm:Furuta_pair}
For a compact space $X$ and its closed subspace $Y \subset X$, there is an isomorphism $K(X, Y) \to \KF(X, Y)$.
\end{thm}

The $K$-group $K(X, Y)$ above is formulated by means of vector bundles, rather than Fredholm operators. In the case of $Y = \emptyset$, Theorem \ref{thm:Furuta_pair} gives Theorem \ref{thm:Furuta}. The proof of Theorem \ref{thm:Furuta_pair} is also included in Appendix.

\medskip

Thanks to the above result of Furuta, we have:

\begin{lem} \label{lem:fda_untwisted_first}
For a compact space $X$, the map $\alpha : [X, \F(\H)] \to \KF(X)$ is bijective.\end{lem}

\begin{proof}
For (a), we consider the following diagram:
$$
\xymatrix{
 & [X, \F(\H)] \ar[dl]_{\mathrm{ind}} \ar[dr]^{\alpha} & \\
K(X) \ar[rr] &   & \KF(X),
}
$$
where $\mathrm{ind} : [X, \F(\H)] \to K(X)$ is the isomorphism constructed in \cite{A,A-Se}. The homomorphism $K(X) \to \KF(X)$ is introduced in Subsection \ref{subsec:vectorial_bundle}. The method showing the surjectivity of $\mathrm{ind}$ in \cite{A} allows us to realize any vector bundle $E \to X$ as $E = \bigcup_{x \in X} \Ker \hat{A}_x$ by means of a map $A : X \to \F(\H)$ such that $\sigma(\hat{A}^2_x) = \{ 0, 1 \}$ for all $x \in X$. Thus, the above diagram is commutative, so that Theorem \ref{thm:Furuta} implies the present lemma.
\end{proof}

\begin{lem} \label{lem:fda_untwisted_second}
For $(X, \mathrm{pt}) \in \cC$, the map $\alpha : [(X, \mathrm{pt}), (\F(\H), 1)] \to \KF(X, \mathrm{pt})$ is bijective.
\end{lem}

\begin{proof}
We use the exact sequences for $(X, \mathrm{pt})$. By Proposition \ref{prop:natural_transformation}, the diagram
$$
\begin{CD}
[(I, \partial I), (\F(\H), 1)] @>>>
[(X, \mathrm{pt}), (\F(\H), 1)] @>>>
[X, \F(\H)] \\
@VV{\alpha}V @VV{\alpha}V @VV{\alpha}V \\
\KF(I, \partial I) @>>> \KF(X, \mathrm{pt}) @>>> \KF(X).
\end{CD}
$$
is commutative. Because $\F(\H)$ is a representing space for $K$-theory \cite{A-Se}, we have $\pi_1(\F(\H), 1) = 0$. By Theorem \ref{thm:Furuta_pair}, we also have $\KF(I, \partial I) \cong K(I, \partial I) = 0$. Hence Lemma \ref{lem:fda_untwisted_first} leads to the present lemma.
\end{proof}

The following is also a result of Furuta:

\begin{lem}[\cite{F1}] \label{lem:collaption_KF}
If $X$ is compact and $Y \subset X$ is closed, then the quotient map $q : X \to X/Y$ induces an isomorphism $\KF(X, Y) \cong \KF(X/Y, \mathrm{pt})$.
\end{lem}

\begin{proof}
Under the assumption, the topology of $(X/Y) - \mathrm{pt}$ induced from $X/Y$ coincides with the topology of $X - Y$ induced from $X$. Hence the isomorphism classes in $\cKF(X, Y)$ correspond bijectively to those in $\cKF(X/Y, \mathrm{pt})$ via $q$. Since this correspondence respects homotopies, the lemma is proved.
\end{proof}

\begin{proof}[Proof of Proposition \ref{prop:fda_untwisted}]
The quotient $q : X \to X/Y$ gives the diagram:
$$
\begin{CD}
[(X, Y), (\F(\H), 1)] @>{\alpha}>> \KF(X, Y) \\
@A{q^*}AA @AA{q^*}A \\
[(X/Y, \mathrm{pt}), (\F(\H), 1)] @>>{\alpha}> \KF(X/Y, \mathrm{pt}).
\end{CD}
$$
By Lemma \ref{lem:collaption_KF}, the right $q^*$ is bijective. Since $Y \to X$ is a cofibration, the left $q^*$ is also bijective, and the diagram above is commutative. Now Lemma \ref{lem:fda_untwisted_second} establishes  Proposition \ref{prop:fda_untwisted}.
\end{proof}


\subsection{Main theorem}

\begin{thm} \label{thm:main}
For a CW complex $X$ and a principal $PU(\H)$-bundle $P \to X$, the homomorphism $\alpha_{-n} : \ K_P^{-n}(X) \to \KF_P^{-n}(X)$, ($n \ge 0$) is bijective.
\end{thm}

To prove this theorem, we begin with the case that $X$ is finite.

\begin{lem} \label{lem:main}
Let $X$ be a finite CW complex. Then, for a principal $PU(\H)$-bundle $P \to X$, the homomorphism $\alpha_{-n} : \ K_P^{-n}(X) \to \KF_P^{-n}(X)$, ($n \ge 0$) is bijective. 
\end{lem}

\begin{proof}
We prove this lemma by an induction on the number $r$ of cells in $X$. If $r = 1$, then $X$ consists of a point, so that $\alpha_{-n}$, ($n \ge 0$) is bijective by Proposition \ref{prop:fda_untwisted}. If $r > 1$, then we can express $X$ as $X = e^q \cup Y$, where $e^q$ is a $q$-dimensional cell, and $Y$ is a subcomplex with $(r-1)$ cells. Proposition \ref{prop:natural_transformation} gives the following commutative diagram for $n \ge 0$.
$$
\xymatrix@C=10.5pt{
K^{-n-1}_{P|_Y}(Y) \ar[r] \ar[d] &
K^{-n}_P(X, Y) \ar[r] \ar[d] &
K^{-n}_P(X) \ar[r] \ar[d]^{\alpha_{-n}} &
K^{-n}_{P|_Y}(Y) \ar[r] \ar[d] &
K^{-n+1}_P(X, Y) \ar[d] \\
\KF^{-n-1}_{P|_Y}(Y) \ar[r] &
\KF^{-n}_P(X, Y) \ar[r] &
\KF^{-n}_P(X) \ar[r] &
\KF^{-n}_{P|_Y}(Y) \ar[r] &
\KF^{-n+1}_P(X, Y)
}
$$
As the hypothesis of the induction, we assume the first and forth vertical maps are bijective. The excision axiom gives $K_P^{-n}(X, Y) \cong K_{P|_{D^q}}^{-n}(D^q, S^{q-1})$. Similarly, $\KF_P^{-n}(X, Y) \cong \KF_{P|_{D^q}}^{-n}(D^q, S^{q-1})$. Since every principal $PU(\H)$-bundles over the disk $D^q$ are trivial, Proposition \ref{prop:fda_untwisted} implies that the second and fifth maps are bijective. Therefore the third map is also bijective. Notice that the five lemma works even if the lower row is not exact at $\KF_{P|_Y}^0(Y)$.
\end{proof}

\begin{proof}[Proof of Theorem \ref{thm:main}]
Let $X^q \subset X$ be the subcomplex consisting of cells whose dimensions are less than or equal to $q$. Identifying $X^q \times \{ q + 1 \} \subset X^q \times [q, q+1]$ with $X^q \times \{ q + 1 \} \subset X^{q+1} \times [q+1, q+2]$, we get the ``telescope'' $\tilde{X}$ of $X$:
$$
\tilde{X} = 
X^0 \times [0, 1] \cup X^1 \times [1, 2] \cup X^2 \times [2, 3] \cup \cdots.
$$
In a similar way, we get a principal $PU(\H)$-bundle $\tilde{P} \to \tilde{X}$ from $P \to X$. As is known \cite{M}, the projections $X^q \times [q, q+1] \to X^q$ induce a homotopy equivalence $\varpi : \tilde{X} \to X$. Since $\varpi^*P \cong \tilde{P}$, we have $K_{\tilde{P}}^{-n}(\tilde{X}) \cong K_P^{-n}(X)$ as well as $\KF_{\tilde{P}}^{-n}(\tilde{X}) \cong \KF_P^{-n}(X)$. Thus, to show Theorem \ref{thm:main}, it suffices to prove that $\alpha_{-n} : K_{\tilde{P}}^{-n}(\tilde{X}) \to \KF_{\tilde{P}}^{-n}(\tilde{X})$ is bijective for $n \ge 0$. For this purpose, we let $\tilde{Y}$ be the subcomplex in $\tilde{X}$ given by $\tilde{Y} = \coprod_q X^q = \coprod_q X^q \times \{ q \}$. For $n \ge 0$, Proposition \ref{prop:natural_transformation} gives:
$$
\xymatrix@C=10.5pt{
K^{-n-1}_{\tilde{P}|_{\tilde{Y}}}(\tilde{Y}) \ar[r] \ar[d] &
K^{-n}_{\tilde{P}}(\tilde{X}, \tilde{Y}) \ar[r] \ar[d] &
K^{-n}_{\tilde{P}}(\tilde{X}) \ar[r] \ar[d]^{\alpha_{-n}} &
K^{-n}_{\tilde{P}|_{\tilde{Y}}}(\tilde{Y}) \ar[r] \ar[d] &
K^{-n+1}_{\tilde{P}}(\tilde{X}, \tilde{Y}) \ar[d] \\
\KF^{-n-1}_{\tilde{P}|_{\tilde{Y}}}(\tilde{Y}) \ar[r] &
\KF^{-n}_{\tilde{P}}(\tilde{X}, \tilde{Y}) \ar[r] &
\KF^{-n}_{\tilde{P}}(\tilde{X}) \ar[r] &
\KF^{-n}_{\tilde{P}|_{\tilde{Y}}}(\tilde{Y}) \ar[r] &
\KF^{-n+1}_{\tilde{P}}(\tilde{X}, \tilde{Y}).
}
$$
The first and forth columns in the commutative diagram above are bijective: the additivity axiom in Proposition \ref{prop:twisted_K_cohomology} implies
$$
K^{-n}_{\tilde{P}|_{\tilde{Y}}}(\tilde{Y}) \cong
K^{-n}_{\coprod_q P|_{X^q}} (\tcoprod_q X^q) \cong
\prod_q K^{-n}_{P|_{X^q}}(X^q).
$$
Similarly, we have $\KF^{-n}_{\tilde{P}|_{\tilde{Y}}}(\tilde{Y}) \cong \prod_q \KF^{-n}_{P|_{X^q}}(X^q)$. Because $X^q$ is a finite CW complex, the map $\alpha_{-n} : K^{-n}_{P|_{X^q}}(X^q) \to \KF^{-n}_{P|_{X^q}}(X^q)$ is bijective by Lemma \ref{lem:main}, and so is $\alpha_{-n} : \prod_q K^{-n}_{P|_{X^q}}(X^q) \to \prod_q \KF^{-n}_{P|_{X^q}}(X^q)$. The second and fifth columns can be shown to be bijective in the same way, since we have
\begin{align*}
K^{-n}_{\tilde{P}}(\tilde{X}, \tilde{Y}) 
&\cong
K^{-n}_{\coprod_q P|_{X^q} \times I}
(\tcoprod_q X^q \times I, \ \tcoprod_q X^q \times \partial I) \\
&\cong
\prod_q 
K^{-n}_{P|_{X^q} \times I}(X^q \times I, \ X^q \times \partial I) 
=
\prod_q K^{-n-1}_{P|_{X^q}}(X^q).
\end{align*}
Now, the five lemma leads to the bijectivity of the third.
\end{proof}

We have Theorem \ref{ithm:main} by setting $n = 0$ in Theorem \ref{thm:main}. We also have:

\begin{cor}[Bott periodicity]
Under the assumption in Theorem \ref{thm:main}, there is a natural isomorphism $\KF_P^{-n}(X) \cong \KF_P^{-n-2}(X)$ for $n \ge 0$.
\end{cor}

Thus, on CW complexes, we can extend Proposition \ref{prop:cohomology_KF} to get a cohomology theory $\{ \KF_P^n(X, Y) \}_{n \in \Z}$ equivalent to twisted $K$-cohomology.




\appendix

\section{Proof of Furuta's theorem}

We provide proof of Furuta's theorems (Theorem \ref{thm:Furuta} and \ref{thm:Furuta_pair}) culling from \cite{F2}. Most parts of this appendix is devoted to the proof of Theorem \ref{thm:Furuta}, since the proof of Theorem \ref{thm:Furuta_pair} is almost the same. To prove Theorem \ref{thm:Furuta}, we begin with preliminaries in Subsection \ref{subsec:appendix_preliminary}, and then construct a vector bundle from a vectorial bundle in Subsection \ref{subsec:appendix_vector_bundle}. We use the vector bundle to prove that the map $K(X) \to \KF(X)$ is surjective in Subsection \ref{subsec:appendix_surjective}. Finally, the injectivity of $K(X) \to \KF(X)$ is shown in Subsection \ref{subsec:appendix_injective}.


\subsection{Preliminary}
\label{subsec:appendix_preliminary}

Let $\mathbb{E} = (\{ U_\alpha \}_{\alpha \in \A}, (E_\alpha, h_\alpha), \phi_{\alpha \beta}) \in \cKF(X)$ be a vectorial bundle on a compact space $X$. Taking a finer open cover if necessary, we can assume that $E_\alpha$ is a trivial bundle $E_\alpha = U_\alpha \times V_\alpha$, where $V_\alpha = V_\alpha^0 \oplus V_\alpha^1$ is a $\Z_2$-graded Hermitian vector space of finite rank. Since $X$ is compact, we can also assume that $\{ U_\alpha \}_{\alpha \in \A}$ is a finite cover of $X$. For $x \in X$, we put $\A(x) = \{ \alpha \in \A |\ x \in U_\alpha \}$. 

\begin{lem}
There is a positive number $\lambda$ such that: for $x \in X$ and $\alpha, \beta \in \A(x)$, the map $(\phi_{\alpha \beta})_x : V_\beta \to V_\alpha$ induces an isomorphism 
$$
(V_\beta, (h_\beta)_x)_{< \lambda} \cong
(V_\alpha, (h_\alpha)_x)_{< \lambda}.
$$
\end{lem}

\begin{proof}
By the definition of vectorial bundles, we can find, for each point $x \in X$, an open neighborhood $U_x$ of $x$ and a positive number $\lambda_x$ such that: for $y \in U_x$ and $\alpha, \beta \in \A(x)$, the map $(\phi_{\alpha \beta})_y$ induces an isomorphism $(V_\beta, (h_\beta)_y)_{< \lambda_x} \cong (V_\alpha, (h_\alpha)_y)_{< \lambda_x}$. Because $X$ is compact, we can choose a finite number of points $x_1, \ldots, x_n \in X$ so that $U_{x_1}, \ldots, U_{x_n}$ cover $X$. The minimum among $\lambda_{x_1}, \ldots, \lambda_{x_n}$ gives the $\lambda$.
\end{proof}

We choose and fix a positive number $\lambda$ in the lemma above. Then, for $x \in X$, we define a $\Z_2$-graded Hermitian vector space $(E)_x$ to be
$$
(E)_x = 
\coprod_{\alpha \in \A(x)} (V_{\alpha}, (h_\alpha)_x)_{<\lambda}/\sim,
$$
where $\sim$ is an equivalence relation: for $v_\alpha \in V_\alpha$ and $v_\beta \in V_\beta$, we have $v_\alpha \sim v_\beta$ if and only if $(\phi_{\alpha \beta})_x v_\beta = v_\alpha$. We also define a Hermitian map of degree $1$
$$
(h_E)_x : (E)_x \to (E)_x, \quad [v_\alpha] \mapsto [(h_\alpha)_x v_\alpha],
$$
where $[v_\alpha]$ stands for the element in $(E)_x$ represented by $v_\alpha \in (V_\alpha, (h_\alpha)_x)_{<\lambda}$. 

\begin{lem}
For a point $x_0 \in X$ and a number $\mu$ such that $\mu \in (0, \lambda) \cap \rho((h_E)^2_{x_0})$, there is an open neighborhood $U$ of $x_0$ on which the family of vector spaces 
$$
\bigcup_{x \in U} ((E)_x, (h_E)_x)_{< \mu}
$$ 
gives rise to a vector bundle.
\end{lem}

We remark that the family of vector spaces $\bigcup_{x \in X} (E)_x$ is not generally a vector bundle over $X$ since the dimension of $(E)_x$ may jump as $x$ varies.

\begin{proof}
We take and fix $\alpha \in \A(x_0)$. The same argument as in Subsection \ref{subsec:fda} implies that there is an open neighborhood $U$ of $x_0$ on which $\bigcup_{x \in U} (V_\alpha, (h_\alpha)_x)_{<\mu}$ gives rise to a vector bundle over $U$. Hence the natural bijection between $\bigcup_{x \in U} (V_\alpha, (h_\alpha)_x)_{<\mu}$ and $\bigcup_{x \in U} ((E)_x, (h_E)_x)_{< \mu}$ establishes the lemma. 
\end{proof}


\subsection{Construction of vector bundle}
\label{subsec:appendix_vector_bundle}

Let $\{ \rho_0^2, \rho_\infty^2 \}$ be a partition of unity subordinate to the open cover $\{ [0, \lambda), (0, \infty) \}$ of $[0, \infty)$. For $x \in X$ and $\alpha \in \A(x)$, the functional calculus induces the Hermitian map $\rho_0^2((h_\alpha)_x^2) : V_\alpha \to V_\alpha$. We can think of $\rho_0^2((h_\alpha)^2_x)$ as an ``approximation'' of the orthogonal projection onto $(V_\alpha, (h_\alpha)_x)_{<\lambda}$. In fact, the image of $V_\alpha$ under $\rho_0^2((h_\alpha)^2_x)$ is in $(V_\alpha, (h_\alpha)_x)_{<\lambda}$. In particular, the image of the odd part $V_\alpha^1$ is in $(V_\alpha, (h_\alpha)_x)_{<\lambda}^1 = V_\alpha^1 \cap (V_\alpha, (h_\alpha)_x)_{<\lambda}$, since $\rho_0^2((h_\alpha)^2_x)$ is of degree $0$. 

\medskip

Now, we use a partition of unity $\{ \varrho_\alpha \}_{\alpha \in \A}$ subordinate to the open cover $\{ U_\alpha \}_{\alpha \in \A}$ of $X$ to define the linear map $(g)_x$, ($x \in X$) as follows:
$$
(g)_x : \
\bigoplus_{\alpha \in \A} V_\alpha^1 \longrightarrow 
(E^1)_x, \quad
\tbigoplus_\alpha v_\alpha^1 \mapsto 
\sum_\alpha \varrho_\alpha(x) [\rho_0^2((h_\alpha)^2_x) v_\alpha^1],
$$
where $(E)_x = (E^0)_x \oplus (E^1)_x$ and $(E^i)_x = \coprod_{\alpha \in \A(x)} (V_{\alpha}, (h_\alpha)_x)^i_{<\lambda}/\sim$. By means of $(g)_x$, we also define the linear maps
$$
(g^i)_x : \ 
(E^i)_x \oplus \bigoplus_{\alpha \in \A} V_\alpha^1
\longrightarrow
(E^1)_x
$$
to be
\begin{align*}
(g^0)_x(v^0 \oplus (\tbigoplus_\alpha v_\alpha^1)) &=
\rho_\infty((h_E)^2_x) (h_E)_x v^0 + 
(g)_x(\tbigoplus_\alpha v_\alpha^1), \\
(g^1)_x(v^1 \oplus (\tbigoplus_\alpha v_\alpha^1)) &=
\rho_\infty((h_E)^2_x) v^1 + (g)_x(\tbigoplus_\alpha v_\alpha^1).
\end{align*}

\begin{lem} \label{lem:appendix_surjectivity}
The maps $(g^0)_x$ and $(g^1)_x$ are surjective at each $x \in X$.
\end{lem}

\begin{proof}
Consider the eigenspace decomposition $(E^1)_x = \bigoplus_\kappa \Ker((h_E)^2_x - \kappa)$. For each $\kappa$, we have $\rho_0^2(\kappa) \neq 0$ or $\rho_\infty^2(\kappa) \neq 0$. In the case of $\rho_0^2(\kappa) \neq 0$, we can see $\Ker((h_E)^2_x - \kappa) \subset (g)_x(V^1_{\alpha})$ for an $\alpha \in \A(x)$ such that $\varrho_\alpha(x) \neq 0$. In the case of $\rho_\infty^2(\kappa) \neq 0$, we clearly have $\Ker((h_E)^2_x - \kappa) \subset (g^1)_x((E^1)_x)$. Since $\kappa \neq 0$ in this case, we also have $\Ker((h_E)^2_x - \kappa) \subset (g^0)_x((E^0)_x)$.
\end{proof}

Accordingly, we have the following exact sequence at each $x \in X$:
$$
\begin{CD}
0 @>>> 
(F^i)_x @>>> 
(E^i)_x \oplus 
{\displaystyle \bigoplus_{\alpha \in \A}} V_\alpha^1 @>{(g^i)_x}>>
(E^1)_x @>>>
0,
\end{CD}
$$
where $(F^i)_x = \Ker (g^i)_x$. The exact sequence implies that $\dim (F^i)_x$ is locally constant in $x$, because $\dim (E^0)_x - \dim (E^1)_x$ is. While the family of vector spaces $\bigcup_{x \in X} (E^i)_x$ is not generally a vector bundle, we have:

\begin{prop}
For $i = 0, 1$, the family of vector spaces $F^i = \bigcup_{x \in X} (F^i)_x$ gives rise to a vector bundle over $X$.
\end{prop}

\begin{proof}
We take and fix $x_0 \in X$ and $\alpha_0 \in \A(x_0)$. For $x \in U_{\alpha_0}$, we introduce linear maps as follows:
\begin{align*}
(\tau^i_{\alpha_0})_x &: \ (E^i)_x \longrightarrow V_{\alpha_0}^i, & \quad
[v_\alpha^i] &\mapsto (\phi_{\alpha_0 \alpha})_x v_\alpha \\
(g_{\alpha_0})_x &: \
\bigoplus_{\alpha \in \A} V_\alpha^1 \longrightarrow 
V^1_{\alpha_0}. &
\tbigoplus_\alpha v_\alpha^1 &\mapsto
\sum_\alpha \varrho_\alpha(x) 
(\phi_{\alpha_0 \alpha})_x 
\rho_0^2((h_\alpha)^2_x) v_\alpha^1
\end{align*}
Note that $(g_{\alpha_0})_x = (\tau^1_{\alpha_0})_x \circ (g)_x$. We also introduce
$$
(g^i_{\alpha_0})_x : \ 
V_{\alpha_0}^i \oplus 
\bigoplus_{\alpha \in \A} V_\alpha^1
\longrightarrow
V^1_{\alpha_0}
$$
by setting
\begin{align*}
(g^0_{\alpha_0})_x(v^0 \oplus (\tbigoplus_\alpha v_\alpha^1)) &= 
\rho_\infty((h_{\alpha_0})^2_x) (h_{\alpha_0})_x v^0 
+ (g)_x(\tbigoplus_\alpha v_\alpha^1), \\
(g^1_{\alpha_0})_x(v^1 \oplus (\tbigoplus_\alpha v_\alpha^1)) &=
\rho_\infty((h_{\alpha_0})^2_x) v^1 + (g)_x(\tbigoplus_\alpha v_\alpha^1).
\end{align*}
In the same way as in the proof of Lemma \ref{lem:appendix_surjectivity}, we see that $(g^i_{\alpha_0})_x$ are also surjective for $x \in U_{\alpha_0}$. Now, we have the commutative diagram:
$$
\begin{CD}
0 @>>> 
\Ker(g^i_{\alpha_0})_x @>>> 
V_{\alpha_0}^i \oplus 
{\displaystyle \bigoplus_{\alpha \in \A}} V_\alpha^1 
@>{(g^i_{\alpha_0})_x}>>
V^1_{\alpha_0} @>>>
0 \\
@. 
@AAA 
@A{(\tau^i_{\alpha_0})_x \oplus \mathrm{id}}AA 
@AA{(\tau^1_{\alpha_0})_x}A 
@. \\
0 @>>> 
(F^i)_x @>>> 
(E^i)_x \oplus 
{\displaystyle \bigoplus_{\alpha \in \A}} V_\alpha^1 @>{(g^i)_x}>>
(E^1)_x @>>>
0.
\end{CD}
$$
Since $(\tau^i_{\alpha_0})_x$ is injective, so is the map $(F^i)_x \to \Ker(g^i_{\alpha_0})_x$. Because $\dim (E^0)_x - \dim (E^1)_x = \dim V^0_{\alpha_0} - \dim V^1_{\alpha_0}$, we have $\dim (F^i)_x = \dim \Ker(g^i_{\alpha_0})_x$. This implies that the map $(F^i)_x \to \Ker(g^i_{\alpha_0})_x$ is bijective. Consequently, we can identify $F^i|_{U_{\alpha_0}}$ with the family of vector spaces $\bigcup_{x \in U_{\alpha_0}} \Ker(g^i_{\alpha_0})_x$. Because $(g^i_{\alpha_0})_x$ is continuous in $x \in U_{\alpha_0}$, the family $\bigcup_{x \in U_{\alpha_0}} \Ker(g^i_{\alpha_0})_x$ becomes a vector bundle. Hence the identification makes $F^i$ into a vector bundle.
\end{proof}


\subsection{Surjectivity}
\label{subsec:appendix_surjective}

In this subsection, we prove:

\begin{prop} \label{prop:appendix_surjective}
If $X$ is compact, then $K(X) \to \KF(X)$ is surjective.
\end{prop}

So far, a $\Z_2$-graded vector bundle $F = F^0 \oplus F^1$ over $X$ is constructed from a given vectorial bundle $\mathbb{E} \in \cKF(X)$. To prove the proposition above, it suffices to introduce a Hermitian metric and a Hermitian map $h$ on $F$ so that $(F, h)$ is isomorphic to $\mathbb{E}$ as a vectorial bundle. 

\medskip

For $x \in X$, we define a Hermitian metric on $(E^i)_x \oplus (\bigoplus_{\alpha \in \A} V^1_\alpha)$ by:
\begin{align*}
\lVert v^0 \oplus (\tbigoplus_\alpha v^1_\alpha) \rVert^2 &=
\lVert v^0 \rVert^2 
+ \frac{1}{\lambda} \sum_\alpha \lVert v_\alpha \rVert^2, \\
\lVert v^1 \oplus (\tbigoplus_\alpha v^1_\alpha) \rVert^2 &=
\lVert v^1 \rVert^2 
+ \sum_\alpha \lVert v_\alpha \rVert^2.
\end{align*}
We induce the Hermitian metric on $(F^i)_x$ by restriction. We then define
$$
(h_{10})_x : \ (F^0)_x \longrightarrow (F^1)_x, \quad
v^0 \oplus (\tbigoplus_\alpha v^1_\alpha) \mapsto 
(h_E)_x v^0 \oplus (\tbigoplus_\alpha v^1_\alpha).
$$
We put $h_x = (h_{10})_x + (h_{01})_x$, where $(h_{01})_x$ is the adjoint of $(h_{10})_x$:
$$
(h_{01})_x : \ (F^1)_x \longrightarrow (F^0)_x, \quad
v^1 \oplus (\tbigoplus_\alpha v^1_\alpha) \mapsto 
(h_E)_x v^1 \oplus (\tbigoplus_\alpha \lambda v^1_\alpha).
$$

\begin{lem} \label{lem:appendix_identification}
Let $\mu_0 > 0$ be such that $\varrho_0(r) = 1$ for $r \in [0, \mu_0] $. Then for $x \in X$ the composition of the inclusion and the projection
$$
(F^i)_x \longrightarrow 
(E^i)_x \oplus \bigoplus_{\alpha \in \A} V^1_\alpha \longrightarrow 
(E^i)_x
$$
induces an isomorphism $(F, h)_{x, <\mu_0} \cong ((E)_x, (h_E)^2_x)_{< \mu_0}$.
\end{lem}

\begin{proof}
The point $x \in X$ will be fixed in this proof. So we omit subscripts $x$ from $(F^i)_x$, $(E^i)_x$ $(h_E)_x^2$, and so on. The map $h_E^2 \oplus (\bigoplus_{\alpha} h_\alpha^2)$ is Hermitian with respect to the Hermitian metric on $E^i \oplus (\bigoplus_{\alpha \in \A} V^1_\alpha)$ introduced just before this lemma. Hence we get the orthogonal decomposition
$$
E^i \oplus \bigoplus_{\alpha \in \A} V^1_\alpha
= \bigoplus_{\mu \ge 0} \hat{F}_\mu, \quad \quad
\hat{F}_\mu =
\Ker( h_E^2 \oplus (\tbigoplus_\alpha h_\alpha^2) - \mu ).
$$
Since $g^i \circ \left( h_E^2 \oplus (\bigoplus_{\alpha} h_\alpha^2) \right) = h_E^2 \circ g^i$, we also get the orthogonal decomposition
$$
F = \bigoplus_{\mu \ge 0} F_\mu, \quad \quad
F_\mu = F \cap \hat{F}_\mu.
$$
Because $h \circ \left( h_E^2 \oplus (\bigoplus_{\alpha} h_\alpha^2) \right) = \left( h_E^2 \oplus (\bigoplus_{\alpha} h_\alpha^2) \right) \circ h$, the map $h$ preserves the orthogonal decomposition of $F$. Thus, in the following, we will verify the present lemma on each space $F_\mu$.

First, we suppose $\mu \ge \mu_0$. Then we have $(F, h)_{<\mu_0} \cap F_\mu = \{ 0 \}$. To see this, notice $\lambda \ge \mu_0$. For a vector $v^0 \oplus (\bigoplus_\alpha v^1_\alpha)$ in the even part $F^0_\mu$ of $F_\mu$, we have
$$
\frac{ \lVert h_{10}( v^0 \oplus (\tbigoplus_\alpha v^1_\alpha)) \rVert^2 }
{ \lVert v^0 \oplus (\tbigoplus_\alpha v^1_\alpha) \rVert^2 }
\ge
\frac
{ \mu \lVert v^0 \rVert^2 + \sum_\alpha \lVert v^1_\alpha \rVert^2 }
{ \lVert v^0 \rVert^2 + 
\frac{1}{\lambda} \sum_\alpha \lvert v^1_\alpha \rVert^2 }
\ge \min\{ \mu, \lambda \} \ge \mu_0.
$$
Hence the eigenvalue of $h^2$ is greater than or equal to $\mu_0$ on the even part $F^0_\mu$, and so is on the odd part $F^1_\mu$. 

Next, we consider the case of $\mu < \mu_0$. Because $\varrho_0(\mu) = 1$ and $\varrho_\infty(\mu) = 0$, a vector $v^i \oplus (\bigoplus_\alpha v^1_\alpha)$ in $\hat{F}^i_\mu$ belongs to $F^i_\mu$ if and only if $\sum_\alpha \varrho_\alpha v_\alpha^1 = 0$. Hence $F^i_\mu$ has the orthogonal decomposition $F^i_\mu = E^i_\mu \oplus V^1_\mu$, where 
\begin{align*}
E^i_\mu &= E^i \cap \Ker(h_E^2 - \mu), \\
V^1_\mu &= \{
\tbigoplus_\alpha v^1_\alpha \in \bigoplus_{\alpha \in \A}V^1_\alpha |\
\sum_\alpha \varrho_\alpha v^1_\alpha = 0\}
\cap \Ker(\tbigoplus_\alpha h_\alpha^2 - \mu).
\end{align*}
The Hermitian map $h^2$ preserves the decomposition. In particular, $h^2 = \lambda$ on $V^1_\mu$, so that $(F, h)_{<\mu_0} \cap F_\mu = E_\mu$. Thus, $(F, h)_{<\mu_0} = (E, h_E)_{<\mu_0}$.
\end{proof}

The isomorphism in Lemma \ref{lem:appendix_identification} is compatible with the Hermitian maps $(h)_x$ and $(h_E)_x$. In addition, the isomorphism in Lemma \ref{lem:appendix_identification} induces an isomorphism of vector bundles locally, provided that $\mu_0$ is chosen suitably. Now, the construction of $(E)_x$ implies that $(F, h)$ is isomorphic to $\mathbb{E}$ as a vectorial bundle, which completes the proof of Proposition \ref{prop:appendix_surjective}.

\medskip

If $Y \subset X$ is a closed subspace and $\mathrm{Supp} \mathbb{E} \cap Y = \emptyset$, then we also have $\mathrm{Supp} (F, h) \cap Y = \emptyset$. This means that we have the pair of vector bundle $(F^0, F^1)$ and $h_{10} : F^0 \to F^1$ is invertible on $Y$. As is known \cite{A}, such data $(F^0, F^1, h_{10})$ constitute the $K$-group $K(X, Y)$. Thus we get:

\begin{prop}
For a compact space $X$ and its closed subspace $Y \subset X$, there is a surjection $K(X, Y) \to \KF(X, Y)$.
\end{prop}


\subsection{Injectivity}
\label{subsec:appendix_injective}

\begin{prop} 
If $X$ is compact, then $K(X) \to \KF(X)$ is injective.
\end{prop}

\begin{proof}
Suppose that two $\Z_2$-graded Hermitian vector bundles $F_0$ and $F_1$ over $X$ give the same element in $\KF(X)$. Then there is a vectorial bundle $\tilde{\mathbb{F}}$ on $X \times [0, 1]$ such that $\tilde{\mathbb{F}}|_{X \times \{ i \}}$ is isomorphic to $(F_i, h_i)$ as a vectorial bundle, where $h_i$ is the trivial Hermitian map $h_i = 0$. Thanks to the construction proving Proposition \ref{prop:appendix_surjective}, we can replace $\tilde{\mathbb{F}}$ by a pair $(\tilde{F}, \tilde{h})$, where $\tilde{F}$ is a $\Z_2$-graded Hermitian vector bundle on $X \times [0, 1]$ and $\tilde{h} : \tilde{F} \to \tilde{F}$ is a Hermitian map of degree $1$. That $(F_i, h_i) \cong (\tilde{F}, \tilde{h})|_{X \times \{ i \}}$ as vectorial bundles means that $F_i \cong \Ker \tilde{h}^2|_{X \times \{ i \}}$ as vector bundles. This implies that the pairs $(F_i^0, F_i^1)$ and $(\tilde{F}^0|_{X \times \{ i \}}, \tilde{F}^1|_{X \times \{ i \}})$ are in the same class in $K(X)$. Therefore the pairs $(F_0^0, F_0^1)$ and $(F_1^0, F_1^1)$ represent the same class in $K(X)$.
\end{proof}

\medskip

\begin{prop} \label{prop:appendix_injective_relative}
For a compact space $X$ and its closed subspace $Y \subset X$, the map $K(X, Y) \to \KF(X, Y)$ is injective.
\end{prop}

\begin{lem} \label{lem:appendix_injective}
For $i = 0, 1$, let $F_i = F_i^0 \oplus F_i^1$ be a $\Z_2$-graded Hermitian vector bundle over a compact space $X$, and $h_i : F_i \to F_i$ a Hermitian map of degree $1$. If $(F_0, h_0)$ and $(F_1, h_1)$ are isomorphic in $\cKF(X)$, then $F_0^0 \oplus F_1^1$ and $F_0^1 \oplus F_1^0$ are isomorphic as vector bundles.
\end{lem}

\begin{proof}
By the definition of the equivalences in $\cKF(X)$, we have a map $g : F_0 \to F_1$ of degree $0$ compatible with $h_0$ and $h_1$ such that: for each $x \in X$, there are a positive integer $\lambda_x$ and an open neighborhood $U_x$ of $x$ such that: $(g)_y$ induces an isomorphism $(F_0, h_0)_{y, <\lambda_x} \cong (F_1, h_1)_{y, <\lambda_x}$ for all $y \in X$. Since $X$ is compact, we can find a positive number $\lambda$ such that: $(g)_x$ induces an isomorphism $(F_0, h_0)_{x, <\lambda} \cong (F_1, h_1)_{x, <\lambda}$ for each $x \in X$. 

We choose such $\lambda$ as above, and take a partition of unity $\{ \rho_0^2, \rho_\infty^2 \}$ subordinate to the open cover $\{ [0, \lambda), (0, \infty) \}$ of $[0, \infty)$. The Hermitian map $\rho_0((h_1)^2_x) : (F_1)_x \to (F_1)_x$ plays a role of a continuous ``approximation'' of (a square root of) the orthogonal projection onto $(F_1, h_1)_{x, <\lambda}$. In fact, the image of $\rho_0((h_1)^2_x)$ is in $(F_1, h_1)_{x, <\lambda}$. Hence the map $(g)_x^{-1} \rho_0((h_1)^2_x) : (F_1)_x \to (F_0)_x$ makes sense. In particular, this map is continuous in $x$, so that we obtain a vector bundle map $g^{-1}\rho_0(h_1^2) : F_1 \to F_0$ of degree $0$. Similarly, the image of $\rho_\infty((h_1)^2_x)$ is in $(F_1, h_1)_{x, \ge \lambda} = \bigoplus_{\kappa \ge \lambda}\Ker((h_1)^2_x - \kappa)$. Thus, we obtain a vector bundle map $\mathrm{sgn}(h_1) \rho_\infty(h_1^2) : F_1 \to F_1$ of degree $1$, where $\mathrm{sgn}(t) = t/\lvert t \rvert$.

Now, we define $\tilde{h} : F_0 \oplus F_1 \to F_0 \oplus F_1$ to be $\tilde{h} = \tilde{h}_0 + \tilde{h}_\infty$, where
\begin{align*}
\tilde{h}_0 &= 
g\rho_0(h_0^2) + g^{-1}\rho_0(h_1^2), \\
\tilde{h}_\infty &= 
\mathrm{sgn}(h_0)\rho_\infty(h_0^2) - \mathrm{sgn}(h_1)\rho_\infty(h_1^2).
\end{align*}
Then $\tilde{h}_0^2 = \rho_0(h_1^2)^2 + \rho_0(h_0^2)^2$, $\tilde{h}_\infty^2 = \rho_\infty(h_1^2)^2 + \rho_\infty(h_0^2)^2$ and $\tilde{h}_0\tilde{h}_\infty + \tilde{h}_\infty\tilde{h}_0 = 0$. Therefore $\tilde{h}^2 = 1$ and  $\tilde{h}$ is an isomorphism. By construction, $\tilde{h}$ carries the component $F_0^0 \oplus F_1^1$ to $F_0^1 \oplus F_1^0$, and vice verse.
\end{proof}

\begin{proof}[Proof of Proposition \ref{prop:appendix_injective_relative}.]
For $i = 0, 1$, we let $(F_i, h_i)$ represent an element in $K(X, Y)$, where $F_i = F_i^0 \oplus F_i^1$ is a $\Z_2$-graded Hermitian vector bundle on $X$ and $h_i : F_i \to F_i$ is a Hermitian map of degree $1$ such that $h_i$ is invertible on $Y$. Suppose that $(F_i, h_i)$ define the same element in $\KF(X, Y)$. Because $X \times [0, 1]$ is compact, the construction showing Proposition \ref{prop:appendix_surjective} gives a $\Z_2$-graded Hermitian vector bundle $\tilde{F}$ on $X \times [0, 1]$ and a Hermitian map $\tilde{h}$ of degree $1$ such that: $\tilde{h}$ is invertible on $Y \times [0, 1]$ and we have $(F_i, h_i) \cong (\tilde{F}, \tilde{h})|_{X \times \{ i \}}$ in $\cKF(X, Y)$. Now, Lemma \ref{lem:appendix_injective} implies that $(F_i, h_i)$ and $(\tilde{F}, \tilde{h})|_{X \times \{ i \}}$ represent the same class in $K(X, Y)$. Thus $(F_0, h_0)$ and $(F_1, h_1)$ are in the same class in $K(X, Y)$.
\end{proof}


\end{document}